\documentclass[12pt,reqno]{amsart}
\usepackage{amsmath} 
\usepackage{amssymb}
\usepackage{dsfont}
\usepackage[dvips,draft,final]{graphics}
\usepackage[T1]{fontenc}
\usepackage{lmodern}
\usepackage{fancyhdr}
\usepackage{url}
\usepackage{enumerate}
\usepackage{color}
\usepackage{graphicx}
\usepackage{mathrsfs}

\newtheorem{theorem}{Theorem}[section]
\newtheorem{proposition}[theorem]{Proposition}
\newtheorem{lemma}[theorem]{Lemma}

\theoremstyle{definition}
\newtheorem{remark}[theorem]{Remark}

\newcommand{\bel}{\begin{equation} \label}
\newcommand{\ee}{\end{equation}}

\newcommand{\pd}{\partial}

\newcommand{\supp}{{\text{supp}}}

\newcommand{\C}{{\mathbb C}}
\newcommand{\R}{{\mathbb R}}
\newcommand{\N}{{\mathbb N}}

\newcommand{\cB}{{\mathcal B}}

\newcommand{\cF}{{\mathcal F}}

\newcommand{\sH}{{\mathscr H}}
\newcommand{\sZ}{{\mathscr Z}}
\newcommand{\cO}{{\mathcal O}}

\newcommand{\Gi}{{\Gamma}_{\mathrm{in}}}
\newcommand{\Go}{{\Gamma}_{\mathrm{out}}}
\newcommand{\sHi}{{\mathscr{H}}_{\mathrm{in}}}
\newcommand{\bsd}{{\mathrm{BSD}}}
\newcommand{\dist}{{\mathrm{dist}}}
\newcommand{\Sp}{{\mathrm{Sp}}}
\def\epsilon{\varepsilon}
\def\phi {\varphi}

\newcommand{\eps}{{\varepsilon}}
\def\beq{\begin{equation}}
\def\eeq{\end{equation}}
\renewcommand{\leq}{\leqslant}
\renewcommand{\geq}{\geqslant}
\newcommand{\bea}{\begin{eqnarray}}
\newcommand{\eea}{\end{eqnarray}}
\newcommand{\beas}{\begin{eqnarray*}}
\newcommand{\eeas}{\end{eqnarray*}}

\newcommand{\Pim}[1]{\ensuremath{\mathrm{Im} \left( #1 \right)}}
{

\providecommand{\abs}[1]{\left\lvert#1\right\rvert}
\providecommand{\norm}[1]{\left\lVert#1\right\rVert}

\title[Multidimensional Borg-Levinson Inverse Spectral Theory]{Multidimensional Borg-Levinson Inverse Spectral Theory}
\author{\'Eric Soccorsi}

\begin{document}
\begin{abstract}
This text deals with multidimensional Borg-Levinson inverse theory. Its main purpose is to establish that the Dirichlet eigenvalues and Neumann boundary data of the operator $-\Delta +q$, acting in a bounded domain of $\R^d$ with $d \geq 2$, uniquely determine the real-valued bounded potential $q$. We first address the case of incomplete spectral data, where finitely many boundary spectral eigen-pairs remain unknown. Under suitable summability condition on the Neumann data, we also consider the case where only the asymptotic behavior of the eigenvalues is known. Finally, we use the multidimensional Borg-Levinson theory for solving parabolic inverse coefficient problems.\\
\end{abstract}
\maketitle

\tableofcontents

\newpage

\section{A short introduction to inverse spectral problems}
\label{sec-intro}
\setcounter{equation}{0}

Let $\Omega \subset \R^d$, where $d \in \N := \{ 1 , 2, \ldots \}$, be a bounded domain with $C^{1,1}$ boundary $\pd \Omega$. In the particular case where $d=1$, we set $\Omega:=(0,1)$.
Given $q \in L^\infty(\Omega)$, real-valued, we perturb the Dirichlet Laplacian in $L^2(\Omega)$ by $q$, i.e. we consider the operator acting in $L^2(\Omega)$ as $-\Delta+q$, that is endowed with homogeneous Dirichlet boundary conditions.

We investigate the inverse problem of determining the operator $-\Delta+q$, that is of determining the perturbation potential $q$, from knowledge of partial spectral data of $-\Delta+q$. More precisely, we are interested in two types of results: 
\begin{itemize}
\item A {\it uniqueness result}, expressing that every two admissible potentials $q_j$, $j=1,2$, are equal whenever the spectral data of $-\Delta+q_1$ coincide with the ones of $-\Delta+q_2$, i.e. we seek the following implication:
$$ \left( \mbox{Spectral\ data\ of}\ -\Delta+q_1 =  \mbox{Spectral\ data\ of}\ -\Delta+q_2 \right) \Longrightarrow (q_1=q_2). $$
\item A {\it stability result}, claiming that any unknown admissible potential $q$ is  not only uniquely determined (in the sense of the above implication) by the spectral data of $-\Delta+q$, but also that it depends continuously on these data.
\end{itemize} 


\subsection{Self-adjointness, spectral data and all that}
For $M \in (0,+\infty)$ fixed, let $q \in L^\infty(\Omega,\R)$ fulfill
\bel{bb}
\norm{q}_{L^\infty(\Omega)} \leq M.
\ee
We define $A_q$ as the operator in $L^2(\Omega)$, associated with the closed sesquilinear form 
\bel{i0} 
a_q(u,v) := \int_{\Omega} \left( \nabla u(x) \cdot \overline{\nabla v(x)} + q(x) u(x) \overline{v(x)} \right) dx,\ u , v \in D(a_q) := H_0^1(\Omega),
\ee
where $H_0^1(\Omega)$ denotes the closure of $C_0^\infty(\Omega)$, the set of infinitely differentiable and compactly supported functions in $\Omega$, for the topology of the first-order Sobolev space $H^1(\Omega)$.
The operator $A_q$ is self-adjoint in $L^2(\Omega)$ and acts on its {domain}\footnote{The assumption $\pd \Omega \in C^{1,1}$ is needed for applying the classical elliptic regularity theory that establishes that $D(A_q) \subset H^2(\Omega)$.} as 
\bel{i1} 
A_q u = (-\Delta + q) u,\ u \in D(A_q)=H_0^1(\Omega) \cap H^2(\Omega).
\ee
Here, the notation $H^2(\Omega)$ stands for the usual second-order Sobolev space in $\Omega$, and we recall that the graph norm of $A_q$ is equivalent to the one of $H^2(\Omega)$, i.e.
\bel{i2}
c^{-1} \norm{u}_{H^2(\Omega)} \leq \norm{u}_{D(A_q)} := \norm{u}_{L^2(\Omega)} + \norm{A_q u}_{L^2(\Omega)}  \leq c \norm{u}_{H^2(\Omega)},\ u \in D(A_q),
\ee
for some constant $c \in (1,+\infty)$ that depends only on $\Omega$ and $M$. 

Next, since the injection $H_0^1(\Omega) \hookrightarrow L^2(\Omega)$ is compact, then the same is true for the resolvent\footnote{This is provided $0$ is in the resolvent set of $A_q$, but since $q$ is bounded as in \eqref{bb}, one can assume that this is the case without restricting the generality of the reasoning, upon possibly substituting $A_q+M$ for $A_q$.} of $A_q$, and the spectrum of $A_q$ is discrete. We denote by $\{ \lambda_n,\ n \in \N \}$ the non-decreasing sequence of eigenvalues of $A_q$, repeated with their multiplicity, 
$$ \lambda_1 \leq \dots \leq \lambda_n \leq \lambda_{n+1} \leq \ldots,\ n \in \N. $$
In view of \eqref{bb}-\eqref{i0}, we infer from the Min-Max principle that
$$ \lambda_1 \geq -M, $$
and we recall for further use that
\bel{i2a}
\lim_{n \to \infty} \lambda_n = +\infty.
\ee

Let $\{ \varphi_n,\ n \in \N \}$ be an orthonormal basis in $L^2(\Omega)$ of eigenfunctions of $A_q$, such that $$ A_q \varphi_n = \lambda_n \varphi_n,\ n \in \N. $$
With reference to \eqref{i2}-\eqref{i2a}, there exist two constants $n_M \in \N$ and $c \in (0,+\infty)$, both of them depending only on $\Omega$ and $M$, such that
\bel{i2b}
c^{-1} \lambda_n \leq \norm{\varphi_n}_{H^2(\Omega)} \leq c \lambda_n,\ n \geq n_M. 
\ee
Put 
$$\psi_n:= \partial_\nu \varphi_n,\ n \in \N, $$ 
where $\nu$ denotes the outward normal vector to $\pd \Omega$ and $\pd_\nu u := \nabla u \cdot \nu$ is the normal derivative of $\varphi$. Then, it follows from \eqref{i2b} and the continuity of the trace operator $\tau_1 : u \mapsto (\partial_\nu u)_{| \pd \Omega}$ from 
$H^2(\Omega)$ into $H^{1 \slash 2}(\pd \Omega)$, that we have
\bel{i2c}
\norm{\psi_n}_{H^{1 \slash 2}(\pd \Omega)} \leq c  \lambda_n,\ n \geq n_M,
\ee
for some positive constant $c$ that depends only on $\Omega$ and $M$.

\subsection{Review of the one-dimensional case}
Fix $d=1$ and recall that we have $\Omega=(0,1)$ with
$$ A_q=-\frac{d^2}{dx^2} + q(x),\ D(A_q)= \{ u \in H^2(0,1),\ u(0)=u(1)=0 \}, $$
in this case. 

\subsubsection{An obstruction to identifiability}
\label{sec-obs}
A very natural question that arises in this context is to know whether $q$ can be determined by knowledge of $\Sp(A_q) = \{ \lambda_n,\ n \in \N \}$. But the answer is negative as the spectrum does not discriminate between symmetric potentials. This can be seen by noticing that we have
\bel{i3}
U A_q U^{-1} = A_{U q},
\ee
where we have set $(U f)(x):=f(1-x)$ for all $f \in L^2(\Omega)$ and a.e. $x \in \Omega$. Since $U$ is unitary in $L^2(\Omega)$, then the operators $A_q$ and $A_{U q}$ are unitarily equivalent,  by \eqref{i3}. Hence they are iso-spectral: $\Sp(A_{U q}) = \Sp(A_q)$. Thus, one cannot distinguish between the potentials $q$ and $U q$, from knowledge of the two spectra $\Sp(A_q)$ and  $\Sp(A_{U q})$, despite of the fact that $q \neq U q$ when $q$ is not symmetric about the midpoint $x=1 \slash 2$ of the interval $\Omega$. 

Therefore, the spectrum of $A_q$ does not uniquely determine $q$, and some additional spectral data is needed for identifying the potential.

\subsubsection{One-dimensional Borg-Levinson theorem}
Assuming that $\varphi_n'(0)=\frac{d \varphi_n}{dx}(0)=1$ for all $n \in \N$, G. Borg \cite{Borg} and N. Levinson \cite{Levinson} established when $\Sp (A_q)$ is known, that additional knowledge of $\{ \norm{\varphi_n}_{L^2(\Omega)},\ n \in \N \}$ uniquely determines $q$.

\begin{theorem}[Borg (1946) and Levinson (1949)]
\label{thm-BL1}
For $\lambda \in \R$ and for $q_j \in L^\infty(0,1;\R)$, $j=1,2$, let $u_j(\cdot,\lambda)$ be the $H^2(0,1)$-solution to the initial values problem
\bel{sl} 
\left\{ \begin{array}{ll} (-\frac{d^2}{dx^2} + q_j(x)) u_j(x,\lambda) = \lambda u_j(x,\lambda),\  & x \in (0,1)\\
u_j(0,\lambda)=0,\ u_j'(0,\lambda)=1. & \end{array} \right. 
\ee
Denote by $\{ \lambda_{j,n},\ n \in \N \}$ the non-decreasing sequence of the Dirichlet eigenvalues associated with $A_{q_j}$, obtained by imposing:
$$ u_j(1,\lambda_{j,n})=0,\ n \in \N. $$
Then, we have the implication:
$$ \left( \lambda_{1,n} = \lambda_{2,n}\ \mbox{and}\ \norm{u_1(\cdot,\lambda_{1,n})}_{L^2(0,1)}  = \norm{u_2(\cdot,\lambda_{2,n})}_{L^2(0,1)},\ n \in \N \right) 
\Longrightarrow \left( q_1 = q_2\ \mbox{in}\ (0,1) \right). $$
\end{theorem}
Later on, I. M. Gel'fand and B. M. Levitan proved that uniqueness is still valid upon substituting $u_j'(1,\lambda_{j,n})$ for $\norm{u_j(\cdot,\lambda_{j,n})}_{L^2(0,1)}$, $j=1,2$,  in Theorem \ref{thm-BL1}:
\begin{theorem}[Gel'fand-Levitan (1951)]
\label{thm-GL}
Under the conditions of Theorem \ref{thm-BL1} we have:
$$ \left( \lambda_{1,n} = \lambda_{2,n}\ \mbox{and}\ u_1'(1,\lambda_{1,n})  = u_2'(1,\lambda_{2,n}),\ n \in \N \right) 
\Longrightarrow \left( q_1 = q_2\ \mbox{in}\ (0,1) \right). $$
\end{theorem}

\begin{remark}
Let $p$, $q$, $\rho$ be real-valued and bounded functions in $(0,1)$, with $p$ and $\rho$ positive. Introduce the operator
$$ A_{p,q,\rho} = -\rho^{-1} \frac{d}{dx} \left( p \frac{d}{dx} \right) + q,\  D(A_{p,q,\rho})= \{ u \in H_0^1(0,1),\ p u' \in H^1(0,1) \}, $$
which is self-adjoint in $L_\rho^2(0,1)$, the usual $L^2(0,1)$-space endowed with the weighted scalar product $\langle u , v \rangle_{L_\rho^2(0,1)}=\int_0^1 \rho(x) u(x) \overline{v(x)} dx$. Denote by $\{ \lambda_n,\ n \in \N \}$ the sequence of the eigenvalues of $A_{p,q,\rho}$ and by $\{ u_n,\ n \in \N \}$ a $L_\rho^2(0,1)$-orthonormal basis of eigenfunctions of $A_{p,q,\rho}$, obeying $A_{p,q,\rho} u_n = \lambda_n u_n$.
If $p$ and $\rho$ are $C^{1,1}(0,1)$, then the boundary spectral data 
$$\bsd(p,q,\rho)=\{ (\lambda_n , u_n'(1)),\ n \in \N \}$$ 
uniquely determine either of the three coefficients $p$, $q$ and $\rho$, when the two others are known. Indeed, one can check by using the Liouville transformation $y(x)=L^{-1} \int_0^x p^{-1 \slash 2}(t) \rho^{1 \slash 2}(t) dt$, $x \in (0,1)$, where $L=\int_0^1 p^{-1 \slash 2}(t) \rho^{1 \slash 2}(t) dt$, as a coordinate transformation, that the equation $-(pu')'+qu-\lambda \rho u=0$ in $(0,1)$ reduces to its normal form $-u''+Vu=\lambda u$ in $(0,1)$, where $V=V_{p,q,\rho}$ is expressed in terms of $p$, $q$ and $\rho$, while the boundary spectral data is preserved, i.e. $$\bsd(1,V,1)=\bsd(p,q,\rho).$$ Thus, Theorem \ref{thm-GL} yields recovery of $V$ from $\bsd(p,q,\rho)$, hence the result. Notice that this change of coordinates is no longer valid for discontinuous $p$ and $\rho$. We refer the reader to \cite{Ca} for a specific treatment of this problem. 
\end{remark}

All the approaches from G. Borg, N. Levinson or I. M. Gel'fand and B. M. Levitan, were based on highly one-dimensional techniques, but two great ideas emerged in the 80's that paved the way toward solving the multidimensional Borg-Levinson inverse spectral problem. The first one is called the C-property, see \cite{R}. It is due to A. G. Ramm who showed that the set $\{ u_1(\cdot,\lambda) u_2(\cdot,\lambda),\ \lambda \in (0,+\infty) \}$ is dense in $L^1(0,1)$. The second one is called the boundary control method, see \cite{B}. It was established by M. I. Belishev upon applying the boundary controllability theory to the hyperbolic equation associated with the Sturm-Liouville system \eqref{sl}. One common nice feature of these two great ideas is that they apply to higher dimensions $d \geq 2$ as well.


\subsection{Multidimensional identification results}

\subsubsection{Boundary spectral data}
Let us recall that $\{ \lambda_n,\ n \in \N \}$ is the non-decreasing sequence of the eigenvalues of $A_q$ (repeated with the multiplicity), that $\{ \varphi_n,\ n \in \N \}$ is a $L^2(\Omega)$-orthonormal basis of eigenfunctions of $A_q$ such that $A_q \varphi_n = \lambda_n \varphi_n$, and that $\psi_n = \pd_\nu \varphi_n$. We define the boundary spectral data (BSD) of $A_q$, or the BSD associated with $q$, as:
$$ \bsd(q) := \{ (\lambda_n, \psi_n),\ n \in \N \}. $$

\begin{remark}
For all $n \in \N$, one may replace $\varphi_n$ by $e^{i \theta_n} \varphi_n$ with $\theta_n \in \R$, in the above definition. Thus it is clear that the BSD are not defined in a unique way: they depend on the choice of the $L^2(\Omega)$-orthonormal basis $\{ \varphi_n,\ n \in \N\}$ of eigenfunctions of $A_q$.
\end{remark}


\subsubsection{Multidimensional identifiability}

In 1988, it was proved for $d \geq 2$ by A. Nachman, J. Sylvester and G. Uhlmann in \cite{NSU}, and independently by R. Novikov in \cite{N}, that the potential $q$ is uniquely determined by $\bsd(q)$, i.e. that the following implication
$$ ( \bsd(q_1) = \bsd(q_2) ) \Longrightarrow (q_1=q_2), $$
holds for any two suitable potentials $q_j$, $j=1,2$. This result has been improved in several ways by various authors. 

Firstly, H. Isozaki  \cite{I} (see also M. Choulli \cite{C}) extended the result of \cite{NSU, N} when finitely many eigenpairs remain unknown.

\begin{theorem}
\label{thm-BLmd1}
For $j=1,2$, let $q_j \in L^\infty(\Omega,\R)$ and write\footnote{That is to say that 
$\{ \lambda_{j,n},\ n \in \N \}$ is the non-decreasing sequence of the eigenvalues of $A_{q_j}$ and that $\psi_{j,n}=\pd_\nu \varphi_{j,n}$ for all $n \in \N$, where $\{ \varphi_{j,n},\ n \in \N \}$ is a $L^2(\Omega)$-orthonormal basis of eigenvectors of $A_{q_j}$ such that $A_{q_j} \varphi_{j,n} =  \lambda_{j,n} \varphi_{j,n}$.} 
$\bsd(q_j)=\{ (\lambda_{j,n},\psi_{j,n}),\ n \in \N \}$. Then, for all $N \in \N$, we have the following implication: 
$$ \left( (\lambda_{1,n},\psi_{1,n}) =  (\lambda_{2,n},\psi_{2,n}),\ n \geq N \right) \Longrightarrow ( q_1 = q_2 ). $$ 
\end{theorem}

Recently, uniqueness in the determination of $q$ was proved in \cite{CS,KKS} from the knowledge of the asymptotic behavior of $\bsd(q)$ when $n \to +\infty$. 

\begin{theorem}
\label{thm-BLmd2}
Let $q_j$ for $j=1,2$, and the notations, be the same as in Theorem \ref{thm-BLmd1}. Assume that the asymptotics of $\bsd(q_1)$ and $\bsd(q_2)$ coincide, in the sense that
$$ \lim_{n \to \infty} (\lambda_{1,n} -\lambda_{2,n}) =0\ \mbox{and}\ \sum_{n=1}^{+\infty} \norm{\psi_{1,n}-\psi_{2,n}}_{L^2(\pd \Omega)}^2 < \infty. $$
Then, we have $q_1=q_2$ in $\Omega$.
\end{theorem}

The multidimensional Borg-Levinson theorem has been studied in many different kinds of settings\footnote{Such as operators in the divergence form, see e.g. B. Canuto and O. Kavian's paper \cite{CK2}, where two unknown coefficients out of three are simultaneously identified by the $\bsd$, or magnetic Schr\"odinger operators, see e.g. Y. Kian's article \cite{K}.} and it is not quite possible to give an extensive survey of this here, but we shall mention a few results which are relevant for the problem under investigation in this text. In \cite{PS}, L. P\"aiv\"arinta and V. Serov proved identifiability of unbounded potentials $q \in L^p(\Omega,\R)$ for $p > d \slash 2$, by $\bsd(q)$. The case of $p=d \slash 2$, $d \geq 3$, has been studied by V. Pohjola in \cite{P}.  As for Borg-Levinson inverse spectral theory with partial Neumann data, we refer the reader to M. Bellassoued, M. Choulli and M. Yamamoto's article \cite{BCY}, where a log-stability estimate for electric potentials which are known in a neighborhood of the boundary, is established with respect to the BSD measured on an arbitrary non-empty open subset of the boundary\footnote{The strategy that is used in this paper is quite the opposite of the one we apply in the last section of this text for solving parabolic inverse coefficient problems by means of the Borg-Levinson theorem, in the sense that the authors rather derive their spectral stability result from a hyperbolic stability inequality.}. 

\subsubsection{Outline}
The paper is organized as follows. In Section \ref{sec-BLth} we prove Theorems \ref{thm-BLmd1} and \ref{thm-BLmd2}. 
In Subsection \ref{sec-pre} we express the strong solution to the Dirichlet problem for $-\Delta +q - \lambda$, $\lambda \in \C \setminus \Sp (A_q)$, in terms of $\bsd(q)$. Subsection \ref{sec-Iso} contains the proof Isozaki's formula, which is useful for the derivation of Theorems \ref{thm-BLmd1} and \ref{thm-BLmd2}, presented in Subsections \ref{sec-prthm1} and \ref{sec-prthm2}, respectively.
In Subsection \ref{sec-stability} we examine the stability issue of the Borg-Levinson inverse problem under study.
Finally, we derive a parabolic identification result in Section \ref{sec-appl:par}, by means of Theorem \ref{thm-BLmd1}.


\section{Multidimensional Borg-Levinson theory}
\label{sec-BLth}
\setcounter{equation}{0} 
This section contains the proof of the incomplete Borg-Levinson theorem stated in Theorem \ref{thm-BLmd1} and the asymptotic Borg-Levinson theorem stated in Theorem \ref{thm-BLmd2}. We start by establishing several technical results that are needed by the derivation of Theorems \ref{thm-BLmd1} and \ref{thm-BLmd2}.


\subsection{Preliminaries}
\label{sec-pre}
For $q \in L^\infty(\Omega,\R)$, $f \in H^{3 \slash 2}(\pd \Omega)$ and $\lambda \in \C$, we consider the boundary value problem (BVP)
\bel{f1}
\left\{ \begin{array}{ll} 
(-\Delta + q - \lambda) u = 0 & \mbox{in}\ \Omega \\
u = f & \mbox{on}\ \pd \Omega. 
\end{array} \right.
\ee

First, we establish that there exists a unique strong solution\footnote{A strong solution to \eqref{f1} is a solution in $H^2(\Omega)$, which satisfies the equation a. e. in $\Omega$.} to the Cauchy problem \eqref{f1} that can be expressed in terms of $\bsd(q)$.

\begin{lemma}
\label{lm1}
Let $q \in L^\infty(\Omega,\R)$ and $f \in H^{3 \slash 2}(\pd \Omega)$. Then, for each $\lambda \in \C \setminus \Sp(A_q)$ there exists a unique solution $u_\lambda \in H^2(\Omega)$ to \eqref{f1}. Moreover $u_\lambda$ reads
\bel{f2}
u_\lambda = \sum_{n=1}^{+\infty} \frac{\langle f ,  \psi_n \rangle_{L^2(\pd \Omega)}}{\lambda-\lambda_n} \varphi_n\ \mbox{in}\ L^2(\Omega),
\ee
and we have
\bel{f3}
\lim_{\lambda \to -\infty} \norm{u_\lambda}_{L^2(\Omega)}^2 = \lim_{\lambda \to -\infty} \left( \sum_{n=1}^{+\infty} \abs{\frac{\langle f ,  \psi_n \rangle_{L^2(\pd \Omega)}}{\lambda-\lambda_n}}^2 \right)= 0. 
\ee
\end{lemma}

\begin{proof} 
We split the proof into three steps.\\

\noindent {\it Step 1: Existence and uniqueness of the solution to \eqref{f1}.}
Since $f \in H^{3 \slash 2}(\pd \Omega)$ and since the trace operator $\tau_0 : v \mapsto v_{| \pd \Omega}$ is surjective from $H^2(\Omega)$ onto $H^{3 \slash 2}(\pd \Omega)$, then there exists $F \in H^2(\Omega)$ such that $\tau_0 F = f$. Thus, $u_\lambda$ is a solution to \eqref{f1} iff $v_\lambda:=u_\lambda -F$ solves
\bel{f4}
\left\{ \begin{array}{ll} 
(-\Delta + q - \lambda) v = G & \mbox{in}\ \Omega \\
v = 0 & \mbox{on}\ \pd \Omega, 
\end{array} \right.
\ee
with $G:=-(-\Delta+q-\lambda)F \in L^2(\Omega) \in L^2(\Omega)$. Next, $\lambda$ being in the resolvent set of $A_q$, we see that \eqref{f4} admits a unique solution $v_\lambda = (A_q-\lambda)^{-1} G$. Thereofore, $u_\lambda=v_\lambda+ F$ is the unique solution to \eqref{f1}, and $u_\lambda \in H^2(\Omega)$.\\

\noindent {\it Step 2: Proof of \eqref{f2}.} For all $n \in \N$, we have
$$ 0 = \langle (-\Delta+q-\lambda) u_\lambda , \varphi_n \rangle_{L^2(\Omega)} = \int_{\Omega} (-\Delta+q(x)-\lambda) u_\lambda(x) \overline{\varphi_n(x)} dx, $$
whence
\beas
0 & = & -\int_{\pd \Omega} \pd_\nu u_\lambda(x) \overline{\varphi_n(x)} dx + \int_{\pd \Omega} u_\lambda(x) \overline{\psi_n(x)} dx + \int_{\Omega} u_\lambda(x) \overline{(A_q-\overline{\lambda}) \varphi_n(x)} dx \\
& = & \int_{\pd \Omega} f(x) \overline{\psi_n(x)} dx + (\lambda_n-\lambda) \int_{\Omega} u_\lambda(x)  \overline{\varphi_n(x)} dx,
\eeas
by integrating by parts. As a consequence we have
$\langle u_\lambda , \varphi_n \rangle_{L^2(\Omega)} = \frac{\langle f , \psi_n \rangle_{L^2(\pd \Omega)}}{\lambda-\lambda_n}$, so \eqref{f2} follows readily from this and the $L^2(\Omega)$-decomposition $u_\lambda=\sum_{n=1}^{+\infty} \langle u_\lambda , \varphi_n \rangle_{L^2(\Omega)} \varphi_n$.\\

\noindent {\it Step 3: Proof of \eqref{f3}.} With reference to \eqref{bb}-\eqref{i0}, we have for all $n \in \N$,
$$
\lambda_n = \langle A_q \varphi_n , \varphi_n \rangle_{L^2(\Omega)} = a_q(\varphi_n,\varphi_n)
= \int_{\Omega} \abs{\nabla \varphi_n(x)}^2 dx + \int_\Omega q(x) \abs{\varphi_n(x)}^2 dx
\geq - M, 
$$
hence $\Sp(A_q) \subset [-M,+\infty)$. Thus,we see that every $\lambda \in (-\infty,-(1+M)]$ lies in the resolvent set of $A_q$, and that
\bel{f4b}
\abs{\frac{\langle f , \psi_n \rangle_{L^2(\pd \Omega)}}{\lambda-\lambda_n}}^2 \leq \abs{\frac{\langle f , \psi_n \rangle_{L^2(\pd \Omega)}}{1+M+\lambda_n}}^2,\ n \in \N.
\ee
Further, since $\sum_{n=1}^{+\infty} \abs{\frac{\langle f , \psi_n \rangle_{L^2(\pd \Omega)}}{1+M+\lambda_n}}^2 = \norm{u_{-(1+M)}}_{L^2(\Omega)}^2 < \infty$,
by \eqref{f2} and the Parseval theorem, and since
$\lim_{\lambda \to -\infty} \abs{\frac{\langle f , \psi_n \rangle_{L^2(\pd \Omega)}}{\lambda-\lambda_n}}^2 = 0$ for all $n \in \N$, we infer from \eqref{f4b} and the Lebesgue dominated convergence theorem that
$$ \lim_{\lambda \to -\infty} \left(  \sum_{n=1}^{+\infty} 
\abs{\frac{\langle f , \psi_n \rangle_{L^2(\pd \Omega)}}{\lambda-\lambda_n}}^2 \right) = \sum_{n=1}^{+\infty} 
\left( \lim_{\lambda \to -\infty} \abs{\frac{\langle f , \psi_n \rangle_{L^2(\pd \Omega)}}{\lambda-\lambda_n}}^2 \right) = 0. $$
Putting this, together with \eqref{f2} and the Parseval formula, we obtain \eqref{f3}.
\end{proof} 

Notice that the series in \eqref{f2} converges in $L^2(\Omega)$ and not in $H^2(\Omega)$. Therefore, the normal derivative $\pd_\nu u_{\lambda}$ of the solution $u_\lambda$ to \eqref{f1} cannot be obtained directly from \eqref{f2}, by substituting $\psi_n$ for $\varphi_n$ in the right hand side. To achieve this, we need to introduce an additional specific spectral parameter $\mu$, and consider the difference $u_\lambda - u_\mu$, as follows.

\begin{lemma}
\label{lm2}
Let $q$ and $f$ be the same as in Lemma \ref{lm1}. Then, for all $\lambda$ and $\mu$ in $\C \setminus \Sp(A_q)$, we have
\bel{f5}
\pd_\nu (u_\lambda - u_\mu) = (\mu-\lambda) \sum_{n=1}^{+\infty} \frac{\langle f , \psi_n \rangle_{L^2(\pd \Omega)}}{(\lambda-\lambda_n)(\mu-\lambda_n)} \psi_n\ \mbox{in}\ H^{1 \slash 2}(\pd \Omega).
\ee
Here $u_\lambda$ (resp., $u_\mu$) denotes the $H^2(\Omega)$-solution to \eqref{f1} (resp., \eqref{f1} where $\lambda$ is replaced by $\mu$), given by Lemma \ref{f1}.
\end{lemma}
\begin{proof}
In view of \eqref{f1}, we see that $v:=u_\lambda - u_\mu$ solves
\bel{f6}
\left\{ \begin{array}{ll} 
(-\Delta + q - \lambda) v = (\lambda-\mu) u_\mu & \mbox{in}\ \Omega \\
v =  0 & \mbox{on}\ \pd \Omega. 
\end{array} \right.
\ee
Since $\lambda$ is in the resolvent set of $A_q$, \eqref{f6} yields
\bel{f6b} 
v=(\lambda-\mu)(A_q-\lambda)^{-1} u_\mu = (\lambda-\mu) \sum_{n=1}^{+\infty} \frac{\langle u_\mu ,  \varphi_n \rangle_{L^2(\Omega)}}{\lambda_n-\lambda} \varphi_n,
\ee
the series being convergent in $L^2(\Omega)$. Recall that we have 
\bel{f6c}
u_\mu = \sum_{n=1}^{+\infty}  \frac{\langle f , \psi_n \rangle_{L^2(\pd \Omega)}}{\mu-\lambda_n} \varphi_n\ \mbox{in}\ L^2(\Omega),
\ee
upon substituting $\mu$ for $\lambda$ in \eqref{f2}. Putting this together with \eqref{f6b}, we get that
\bel{f7} 
v = (\lambda-\mu) \sum_{n=1}^{+\infty} \frac{\langle f ,  \psi_n \rangle_{L^2(\pd \Omega)}}{(\lambda_n-\lambda)(\mu-\lambda_n)} \varphi_n\ \mbox{in}\ L^2(\Omega),
\ee
Next, since $v \in D(A_q)$ and $A_q v= (\lambda-\mu) u_\mu + \lambda v$, we deduce from \eqref{f6c}-\eqref{f7} that
$$
A_q v 
=(\lambda-\mu) \sum_{n=1}^{+\infty} \frac{\lambda_n  \langle f ,  \psi_n \rangle_{L^2(\pd \Omega)}}{(\lambda_n-\lambda)(\mu-\lambda_n)} \varphi_n\ \mbox{in}\ L^2(\Omega).
$$
Therefore, the series in \eqref{f7} converges for the topology of the norm of $A_q$, hence it converges in $H^2(\Omega)$, according to \eqref{i2}. Finally, we obtain \eqref{f5} from this by invoking the continuity of the trace operator $\tau_1: u \mapsto (\pd_\nu u)_{| \pd \Omega}$ from $H^2(\Omega)$ into $H^{1 \slash 2}(\pd \Omega)$. 
\end{proof}

The next lemma claims for any two real-valued bounded potentials $q_1$ and $q_2$, that the solutions to \eqref{f1} associated with either $q=q_1$ or $q=q_2$, are closed as $\lambda \to -\infty$: in some sense the influence of the potential is dimmed when the spectral parameter $\lambda$ goes to $-\infty$.

\begin{lemma}
\label{lm3}
Let $f \in H^{3 \slash 2}(\pd \Omega)$ and let $q_j \in L^\infty(\Omega,\R)$, $j=1,2$. For $\lambda \in \C \setminus ( \Sp(A_{q_1}) \cup \Sp(A_{q_2}) )$, let $u_{j,\lambda}$ be the solution to \eqref{f1} where $q_j$ is substituted for $q$, which is given by Lemma \ref{lm1}. Then, we have
\bel{f7c}
\lim_{\lambda \to -\infty} \norm{\pd_\nu u_{1,\lambda}- \pd_\nu u_{2,\lambda}}_{L^2(\pd \Omega)} =0.
\ee
\end{lemma}
\begin{proof}
Set $w_\lambda:=u_{1,\lambda}-u_{2,\lambda}$, so we have
$$
\left\{ \begin{array}{ll} 
(-\Delta + q_1 - \lambda) w_\lambda = (q_2-q_1) u_{2,\lambda} & \mbox{in}\ \Omega \\
w_\lambda =  0 & \mbox{on}\ \pd \Omega,
\end{array} \right.
$$
from \eqref{f1}, and hence $w_\lambda=(A_{q_1}-\lambda)^{-1} (q_2-q_1) u_{2,\lambda}$. 
Bearing in mind for all real number $\lambda < -\norm{q_1}_{L^\infty(\Omega)}$, that $\norm{(A_{q_1}-\lambda)^{-1}}_{\cB(L^2(\Omega))} = \dist^{-1}(\lambda,\Sp(A_{q_1})) \leq 1 \slash (\norm{q_1}_{L^\infty(\Omega)}+\lambda)$, we find that
$$ \norm{w_\lambda}_{L^2(\Omega)} \leq \frac{\norm{q_2-q_1}_{L^\infty(\Omega)} \norm{u_{2,\lambda}}_{L^2(\Omega)}}{-\lambda-\norm{q_1}_{L^\infty(\Omega)}},\ \lambda \in \left( -\infty,-\norm{q_1}_{L^\infty(\Omega)} \right).$$
Here and in the remaining part of this text, $\cB(L^2(\Omega))$ denotes the space of linear bounded operators\footnote{The usual norm of $T \in \cB(L^2(\Omega))$ is defined by
$\norm{T}_{\cB(L^2(\Omega))} = \sup_{f \in L^2(\Omega) \setminus \{ 0 \}} \frac{\norm{T f}_{L^2(\Omega)}}{\norm{f}_{L^2(\Omega)}}$.} in $L^2(\Omega)$.
From this and \eqref{f3} it then follows that
\bel{f8}
\lim_{\lambda \to -\infty}  \lambda \norm{w_\lambda}_{L^2(\Omega)} =0.
\ee
Next, since  $A_{q_1} w_\lambda = (q_2-q_1) u_{2,\lambda} + \lambda w_\lambda$, it holds true for every $\lambda < -\norm{q_1}_{L^\infty(\Omega)}$ that
$$ \norm{A_{q_1} w_\lambda}_{L^2(\Omega)} \leq \norm{q_2-q_1}_{L^\infty(\Omega)} \norm{u_{2,\lambda}}_{L^2(\Omega)} - \lambda \norm{w_\lambda}_{L^2(\Omega)},$$
so we get $\lim_{\lambda \to -\infty} \norm{A_{q_1} w_\lambda}_{L^2(\Omega)} =0$, from \eqref{f3} and \eqref{f8}. As a consequence we have
$$ \lim_{\lambda \to -\infty} \left( \norm{w_\lambda}_{L^2(\Omega)} + \norm{A_{q_1} w_\lambda}_{L^2(\Omega)} \right) =0. $$
This and \eqref{i2} entail
$$
\lim_{\lambda \to -\infty} \norm{u_{1,\lambda}-u_{2,\lambda}}_{H^2(\Omega)} =0
$$
which together with the continuity of the trace operator $\tau_1: u \mapsto (\pd_\nu u)_{| \pd \Omega}$ from $H^2(\Omega)$ into $H^{1 \slash 2}(\pd \Omega)$, yield \eqref{f7c}.
\end{proof}


\subsection{Isozaki's asymptotic representation formula}
\label{sec-Iso}

Let $q_j \in L^\infty(\Omega,\R)$ satisfy
\bel{g0}
\norm{q_j}_{L^\infty(\Omega)} \leq M,\ j=1,2,
\ee
for some {\it a priori} fixed constant $M \in (0,+\infty)$. 
In \cite{I}, H. Isozaki gives a simple representation formula, expressing the difference $q_1-q_2$ in terms of the Dirichlet-to-Neumann (DN) operator associated with the BVP obtained by substituting $q_j$ for $q$ in \eqref{f1}.
More precisely, adapting the argument of \cite{I} to fit our aim in this text, we fix $\tau \in (1,+\infty)$ and we consider the BVP \eqref{f1} with $\lambda=\lambda_\tau^+:=(\tau+i)^2$ and $q=q_j$, i.e.
\bel{g1}
\left\{ \begin{array}{ll} (-\Delta + q_j - \lambda_\tau^+) u =0 & \mbox{in}\ \Omega \\ u = f & \mbox{on}\ \pd \Omega. \end{array}
\right.
\ee
We denote by $u_{j,\lambda_\tau^+}$ the $H^2(\Omega)$-solution to \eqref{g1} (for the sake of notational simplicity we drop the dependence of $u_{j,\lambda_\tau^+}$ on $f$). Let us introduce the DN map associated with \eqref{g1}, as
\bel{g2}
\begin{array}{cccc}
\Lambda_{j,\lambda_\tau^+} : & H^{3 \slash 2}(\pd \Omega) & \to & H^{1 \slash 2}(\pd \Omega) \\
& f & \mapsto & \left( \pd_\nu u_{j,\lambda_\tau^+} \right)_{| \pd \Omega}. \end{array}
\ee
Given two test functions $f_\tau^\pm$, we shall make precise below, we aim to link the difference $q_1-q_2$ to the asymptotic behavior of 
\bel{g3} 
S_{\tau} := S_{1,\tau}-S_{2,\tau},\ \mbox{where}\ S_{j,\tau}:= \langle \Lambda_{j,\lambda_\tau^+} f_\tau^+ , f_\tau^- \rangle_{L^2(\pd \Omega)},
\ee
as $\tau \to +\infty$.

\subsubsection{Test functions}
For $\xi \in \R^d$ fixed, and for every $\tau \in (\abs{\xi}+1,+\infty)$, we set $\lambda_\tau^\pm:=(\tau \pm i)^2$, and we seek two functions $f_\tau^\pm$ such that
\bel{g4}
\left\{ \begin{array}{l}
(-\Delta - \lambda_\tau^\pm ) f_\tau^\pm = 0\ \mbox{in}\ \Omega \\
\lim_{\tau \to +\infty} f_\tau^+(x) \overline{f_\tau^-(x)} = e^{-i \xi \cdot x},\ x \in \Omega \\
\sup_{\tau \in (\abs{\xi}+1,+\infty)} \norm{f_\tau^\pm}_{C(\overline{\Omega})} < \infty.
\end{array} \right.
\ee
Here and in the remaining part of this text, the notation $\cdot$ (resp., $\abs{\cdot}$) stands for the Euclidian product (resp., norm) in $\R^d$.

Pick $\eta \in \mathbb{S}^{d-1}$ such that $\xi \cdot \eta = 0$, and put
\bel{g5}
\beta_\tau := \sqrt{1- \frac{\abs{\xi}^2}{4 \tau^2}}\ \mbox{and}\ \eta_\tau^\pm := \beta_\tau  \eta \mp \frac{\xi}{2 \tau},
\ee
in such a way that $\abs{\eta_\tau^\pm}=1$. Then, the two following functions
\bel{g6}
f_\tau^\pm(x):= e^{i(\tau \pm i) \eta_\tau^\pm \cdot x}, \ x \in \Omega,
\ee
fulfill the conditions of \eqref{g4}. As a matter of fact, it can be checked through direct computation from \eqref{g5}-\eqref{g6}, that 
$\Delta f_\tau^\pm=-\lambda_\tau^\pm \abs{\eta_\tau^\pm}^2 f_\tau^\pm = -\lambda_\tau^\pm f_\tau^\pm$ in $\Omega$, that $f_\tau^+(x) \overline{f_\tau^-(x)} = e^{-i \frac{\tau+i}{\tau} \xi \cdot x}$ for all $x \in \Omega$, and that
\bel{g6b}
\abs{f_\tau^\pm(x)} \leq e^{\abs{\eta_\tau^\pm} \abs{x}}\leq e^{\abs{x}},\ x \in \overline{\Omega}. 
\ee
We notice for further use from \eqref{g6b} that the estimate 
\bel{g*}
\norm{f_\tau^\pm}_{L^p(X)}  \leq c_* := \left(1 + \abs{\Omega}^{1 \slash 2} + \abs{\pd \Omega}^{1 \slash 2} \right) \sup_{x \in \overline{\Omega}} e^{\abs{x}},
\ee
holds with $X=\Omega$ or $X=\pd \Omega$, and with $p=2$ or $p=\infty$. Here $\abs{\Omega}$ (resp., $\abs{\pd \Omega}$) denotes the diameter of $\Omega$ (resp., the length of $\pd \Omega$). 

\subsubsection{Probing \eqref{f1} with $f_\tau^\pm$}
\label{sec-star}
For $j=1,2$ and $z \in \C \setminus \Sp(A_{q_j})$, we denote by $u_{j,z}^\pm$ the $H^2(\Omega)$-solution to the BVP \eqref{f1}, where $(q_j,z,f_\tau^\pm)$ is substituted for $(q,\lambda,f)$. The function $u_{j,z}^\pm$ is characterized by
\bel{gg0}
\left\{ \begin{array}{ll} 
(-\Delta + q_j - z) u_{j,z}^\pm = 0 & \mbox{in}\ \Omega \\ u_{j,z}^\pm = f_\tau^\pm & \mbox{on}\ \pd \Omega, \end{array}
\right.
\ee
hence $v_{j,z}^\pm := u_{j,z}^\pm - f_\tau^\pm$ solves
\bel{gg1}
\left\{ \begin{array}{ll} 
(-\Delta + q_j - z) v_{j,z}^\pm = -(-\Delta + q_j - z) f_\tau^{\pm}& \mbox{in}\ \Omega \\ v_{j,z}^\pm = 0 & \mbox{on}\ \pd \Omega.
\end{array}
\right.
\ee
Moreover, since $(-\Delta + q_j - z) f_\tau^\pm=(q_j + \lambda_\tau^\pm - z) f_\tau^\pm$, by the first line in \eqref{g4}, it follows from \eqref{gg1} that
\bel{g6c}
v_{j,z}^\pm=-(A_{q_j} - z)^{-1}(q_j + \lambda_\tau^\pm - z) f_\tau^\pm.
\ee
Let us now examine the case where $z=\lambda_\tau^\pm$, which is permitted since $\lambda_\tau^\pm$ belongs to the resolvent set of the self-adjoint operator $A_{q_j}$, as we have:
\bel{gg2} 
\Pim{\lambda_\tau^\pm}=\pm 2 \tau \neq 0.
\ee
We shall establish that the $L^2(\Omega)$-norm of $v_{j,\lambda_\tau^\pm}^\pm$ scales like $\tau^{-1}$ as $\tau$ becomes large, whereas the one of $u_{j,\lambda_\tau^\pm}^\pm$ is bounded uniformly in $\tau \in (1+\abs{\xi},+\infty)$. 
To do that, we substitute $\lambda_\tau^\pm$ for $z$ in \eqref{g6c} and get that $v_{j,\lambda_\tau^\pm}^\pm =-(A_{q_j} - \lambda_\tau^\pm)^{-1} q_j f_\tau^\pm$. Next, using that
$\norm{(A_{q_j} - \lambda_\tau^\pm)^{-1}}_{\cB(L^2(\Omega))} = {\rm dist}^{-1}(\lambda_\tau^\pm,\Sp(A_{q_j})) \leq (2 \tau)^{-1}$, according to \eqref{gg2}, we obtain
\bel{g6d} 
\norm{v_{j,\lambda_\tau^\pm}^\pm}_{L^2(\Omega)}  \leq \frac{\norm{q_j}_{L^\infty(\Omega)} \norm{f_\tau^\pm}_{L^2(\Omega)}}{2 \tau} \leq \frac{M c_*}{2 \tau},\ j=1,2,
\ee
upon applying \eqref{g0} and \eqref{g*} with $(p,X)=(2, \Omega)$.
Now, bearing in mind that $\tau \geq 1$ and recalling that $u_{j,z}^\pm = v_{j,z}^\pm + f_\tau^\pm$, we derive from \eqref{g6d}  that
\bel{g6e}
\norm{u_{j,\lambda_\tau^\pm}^\pm}_{L^2(\Omega)} \leq \frac{M+2}{2} c_*,\ j=1,2. 
\ee

\subsubsection{Isozaki's formula}

\begin{proposition}
\label{pr-iso}
For $j=1,2$, let $q_j \in L^\infty(\Omega,\R)$ satisfy \eqref{g0}. Then, for all $\xi \in \R^d$, we have
\bel{iso} 
\int_\Omega (q_1(x)-q_2(x)) e^{-i \xi \cdot x} dx = \lim_{\tau \to +\infty} S_\tau,
\ee
where $S_\tau$ is defined by \eqref{g2}-\eqref{g3}.
\end{proposition}
\begin{proof} 
For $j=1,2$, we consider the $H^2(\Omega)$-solution $u_{j,\lambda_\tau^+}^+$ to the BVP \eqref{g1} with $f=f_\tau^+$:
\bel{g7}
\left\{ \begin{array}{ll} 
(-\Delta + q_j - \lambda_\tau^+) u_{j,\lambda_\tau^+}^+ =0 & \mbox{in}\ \Omega \\ u_{j,\lambda_\tau^+}^+ = f_\tau^+ & \mbox{on}\ \pd \Omega. \end{array}
\right.
\ee
Upon left-multiplying the first line of \eqref{g7} by $\overline{f_\tau^-}$, integrating over $\Omega$, and applying the Green formula, we obtain with the aid of \eqref{g4} that
\beas
0 & = & \int_\Omega (-\Delta+q_j-\lambda_\tau^+) u_{j,\lambda_\tau^+}^+(x) \overline{f_\tau^-(x)} dx \\
& = & \langle f_\tau^+ , \pd_\nu f_\tau^- \rangle_{L^2(\pd \Omega)} - \langle \pd_\nu u_{j,\lambda_\tau^+}^+ , f_\tau^- \rangle_{L^2(\pd \Omega)} + \int_\Omega  u_{j,\lambda_\tau^+}^+(x) \overline{(-\Delta+q_j-\lambda_\tau^-) f_\tau^-(x)} dx \\
& = & \langle f_\tau^+ , \pd_\nu f_\tau^- \rangle_{L^2(\pd \Omega)} -S_{j,\tau} +\int_\Omega q_j(x) u_{j,\lambda_\tau^+}^+(x) \overline{f_\tau^-(x)} dx,\ j=1,2.
\eeas
Thus, we have $S_{j,\tau} = \langle f_\tau^+ ,\pd_\nu f_\tau^- \rangle_{L^2(\pd \Omega)} +\int_\Omega q_j(x) u_{j,\lambda_\tau^+}^+(x) \overline{f_\tau^-(x)} dx$ for $j=1,2$, and consequently
\bel{g8}
S_\tau = S_{1,\tau}-S_{2,\tau}=\int_\Omega \left( q_1(x) u_{1,\lambda_\tau^+}^+(x)-q_2(x) u_{2,\lambda_\tau^+}^+(x) \right) \overline{f_\tau^-(x)} dx. 
\ee
Next, taking into account that $u_{j,\lambda_\tau^+}^+=v_{j,\lambda_\tau^+}^+ + f_\tau^+$ for $j=1,2$, we deduce from \eqref{g8} that
$$
S_\tau - \int_\Omega (q_1(x)-q_2(x)) f_\tau^+(x) \overline{f_\tau^-(x)} dx =
\int_\Omega (q_1(x)v_{1,\lambda_\tau^+}^+ (x)-q_2(x) v_{2,\lambda_\tau^+}^+(x)) \overline{f_\tau^-(x)} dx.
$$
Therefore, by applying \eqref{g*} with $(p,X)=(2,\Omega)$ and \eqref{g6d}, we get
$$
\abs{S_\tau - \int_\Omega (q_1(x)-q_2(x)) f_\tau^+(x) \overline{f_\tau^-(x)} dx} \leq 
\left( \sum_{j=1}^2 \norm{q_j}_{L^\infty(\Omega)} \norm{v_{j,\lambda_\tau^+}^+}_{L^2(\Omega)}  \right) \norm{f_\tau^-}_{L^2(\Omega)} \leq \frac{M^2 c_*^2}{\tau},
$$
which leads to:
\bel{g10}
\lim_{\tau \to +\infty} \left( S_\tau - \int_\Omega (q_1(x)-q_2(x)) f_\tau^+(x) \overline{f_\tau^-(x)} dx \right) = 0.
\ee
Finally, as we have
$$ \lim_{\tau \to +\infty} \int_\Omega (q_1(x)-q_2(x)) f_\tau^+(x) \overline{f_\tau^-(x)} dx = \int_\Omega (q_1(x)-q_2(x)) e^{-i \xi \cdot x} dx, $$
by the second line of \eqref{g4}, \eqref{g*} with $(p,X)=(+\infty,\Omega)$, \eqref{g0}, and the dominated convergence theorem, the desired result follows directly from this and from \eqref{g10}. 
\end{proof}

Armed with Proposition \ref{pr-iso} we turn now to proving Theorems  \ref{thm-BLmd1} and \ref{thm-BLmd2}.


\subsection{Proof of the incomplete Borg-Levinson theorem}
\label{sec-prthm1}
In this section we prove Theorem \ref{thm-BLmd1}. In view of Proposition \ref{pr-iso}, we have to show that
\bel{p1}
\lim_{\tau \to +\infty} S_\tau =0,\ \xi \in \R^d.
\ee
Indeed, by combining the Isozaki formula \eqref{iso} with \eqref{p1}, we get for every $\xi \in \R^d$ that
the Fourier transform
\bel{p0}
(\cF q)(\xi) := \frac{1}{(2 \pi)^{d \slash 2}} \int_{\R^d} q(x) e^{-i \xi \cdot x} dx
\ee
of the following function
\bel{p*}
q(x) := \left\{ \begin{array}{cl} q_1(x)-q_2(x) & \mbox{if}\ x \in \Omega \\ 0 & \mbox{if}\  x \in \R^d \setminus \Omega, \end{array} \right.
\ee
reads $(\cF q)(\xi) = \frac{1}{(2 \pi)^{d \slash 2} }\int_{\Omega} (q_1(x)-q_2(x)) e^{-i \xi \cdot x} dx =  \frac{1}{(2 \pi)^{d \slash 2}} \lim_{\tau \to +\infty} S_\tau =0$. By the injectivity of Fourier transform $\cF$, this entails that $q=0$ in $\R^d$,  i.e. that $q_1=q_2$ in $\Omega$. 

We turn now to establishing \eqref{p1}. To this purpose, we fix $\xi \in \R^d$, pick $\tau \in (\abs{\xi}+1,+\infty)$, and for $j=1,2$ and all $z \in \C \setminus \Sp(A_{q_j})$, we consider the $H^2(\Omega)$-solution $u_{j,z}^+$ to the BVP \eqref{f1}, where $(q_j,z,f_\tau^+)$ is substituted for $(q,\lambda,f)$, i.e.
$$
\left\{ \begin{array}{ll} 
(-\Delta + q_j - z) u_{j,z}^+ = 0 & \mbox{in}\ \Omega \\ u_{j,z}^+  = f_\tau^+ & \mbox{on}\ \pd \Omega. \end{array}
\right.
$$
For $z_j \in  \C \setminus \Sp(A_{q_j})$, $j=1,2$, we put $u_{j,z_1,z_2}^+:=u_{j,z_1}^+-u_{j,z_2}^+$ and recall from \eqref{g3} that
$$
S_\tau = \langle \Lambda_{1,\lambda_\tau^+} f_\tau^+ , f_\tau^- \rangle_{L^2(\pd \Omega)}-\langle \Lambda_{2,\lambda_\tau^+} f_\tau^+ , f_\tau^-
 \rangle_{L^2(\pd \Omega)} = \langle \pd_\nu u_{1,\lambda_\tau^+}^+- \pd_\nu u_{2,\lambda_\tau^+}^+ , f_\tau^- \rangle_{L^2(\pd \Omega)}. $$
Thus, for every $\mu \in (-\infty,- M)$ we have
\bel{p2b}
S_\tau = \langle \pd_\nu u_{1,\lambda_\tau^+,\mu}^+ , f_\tau^- \rangle_{L^2(\pd \Omega)} - \langle \pd_\nu u_{2,\lambda_\tau^+,\mu}^+ , f_\tau^- \rangle_{L^2(\pd \Omega)} + \langle \pd_\nu u_{1,\mu}^+- \pd_\nu u_{2,\mu}^+ , f_\tau^- \rangle_{L^2(\pd \Omega)}.
\ee
In view of \eqref{f7c}, we have $\lim_{\mu \to -\infty} \norm{ \pd_\nu u_{1,\mu}^+- \pd_\nu u_{2,\mu}^+}_{L^2(\pd \Omega)}=0$ so we get
\bel{p3b}
S_\tau = \lim_{\mu \to -\infty}  \langle \pd_\nu u_{1,\lambda_\tau^+,\mu}^+ - \pd_\nu u_{2,\lambda_\tau^+,\mu}^+, f_\tau^- \rangle_{L^2(\pd \Omega)},
\ee
upon sending $\mu$ to $-\infty$ in \eqref{p2b}.

Next, we introduce
\bel{p4}
\kappa_{\tau,\mu}(t) : = \frac{\mu-\lambda_\tau^+}{(\lambda_\tau^+-t)(\mu-t)},\ t \in \R \setminus \{ \mu \},
\ee
and  set
\bel{p5}
\zeta_\tau(\psi,\phi) := \langle f_\tau^+ , \psi \rangle_{L^2(\pd \Omega)} \overline{\langle f_\tau^- ,\phi \rangle}_{L^2(\pd \Omega)},\ \psi,\ \phi \in L^2(\pd \Omega).
\ee
In light of Lemma \ref{lm2}, the scalar product in the right hand side of \eqref{p3b} decomposes as
\bel{p3c} 
\langle \pd_\nu u_{1,\lambda_\tau^+,\mu}^+ - \pd_\nu u_{2,\lambda_\tau^+,\mu}^+, f_\tau^- \rangle_{L^2(\pd \Omega)}
=  \sum_{n=1}^{+\infty} \left( \kappa_{\tau,\mu}(\lambda_{1,n}) \zeta_\tau(\psi_{1,n}) -  \kappa_{\tau,\mu}(\lambda_{2,n}) \zeta_\tau(\psi_{2,n}) \right),
\ee
where the notation $\zeta_\tau(\psi)$ is a shorthand for $\zeta_\tau(\psi,\psi)$. Further, as $(\lambda_{1,n},\psi_{1,n})=(\lambda_{2,n},\psi_{2,n})$ for every $n \geq N$, by assumption, \eqref{p3c} becomes
\bel{p6} 
\langle \pd_\nu u_{1,\lambda_\tau^+,\mu}^+ - \pd_\nu u_{2,\lambda_\tau^+,\mu}^+, f_\tau^- \rangle_{L^2(\pd \Omega)}
=  \sum_{n=1}^{N-1} \left( \kappa_{\tau,\mu}(\lambda_{1,n}) \zeta_\tau(\psi_{1,n}) -  \kappa_{\tau,\mu}(\lambda_{2,n}) \zeta_\tau(\psi_{2,n}) \right), 
\ee
the sum in the right hand side of the above equality being taken equal to zero when $N=1$.
Further, taking into account that  
$$\lim_{\mu \to -\infty} \kappa_{\tau,\mu}(\lambda_{j,n}) = 1 \slash (\lambda_\tau^+-\lambda_{j,n}),\ j=1,2,\ n=1,\ldots,N-1, $$
we deduce from \eqref{p3b} and \eqref{p6}, that
$$
S_\tau =  \sum_{n=1}^{N-1} \left( 
\frac{\zeta_\tau(\psi_{1,n}) }{\lambda_\tau^+-\lambda_{1,n}} -  \frac{\zeta_\tau(\psi_{2,n}) }{\lambda_\tau^+-\lambda_{2,n}} \right). 
$$
Next, bearing in mind that $\Pim{\lambda_\tau^+}= 2 \tau$, we see that $\abs{\lambda_\tau^+-\lambda_{j,n}} \geq 2 \tau$ for $j=1,2$ and for all $n=1,\ldots,N-1$, and hence that
\beas
\abs{S_\tau}  & \leq & \sum_{n=1}^{N-1} \left( \abs{\zeta_\tau(\psi_{1,n})}+ \abs{\zeta_\tau(\psi_{2,n})}\right) (2 \tau)^{-1} \\
& \leq & \norm{f_\tau^+}_{L^2(\pd \Omega)} \norm{f_\tau^-}_{L^2(\pd \Omega)} \sum_{n=1}^{N-1} \left( \norm{\psi_{1,n}}_{L^2(\pd \Omega)}^2 + \norm{\psi_{2,n}}_{L^2(\pd \Omega)}^2 \right) (2 \tau)^{-1},
\eeas
from \eqref{p4}-\eqref{p5}.
Now, applying \eqref{g*} with $X=\pd \Omega$ and $p=2$, we obtain that 
$$\abs{S_\tau} \leq c_*^2 \sum_{n=1}^{N-1} \left( \norm{\psi_{1,n}}_{L^2(\pd \Omega)} ^2+ \norm{\psi_{2,n}}_{L^2(\pd \Omega)}^2 \right)  \tau^{-1},$$ 
which immediately entails \eqref{p1}.


\subsection{Proof of the asymptotic Borg-Levinson theorem}
\label{sec-prthm2} 
In this section we prove Theorem \ref{thm-BLmd2}. We stick with the notations of Section \ref{sec-prthm1} and recall from \eqref{p3b} and \eqref{p3c} that
\bel{k1}
S_\tau = \lim_{\mu \to -\infty} \langle \pd_\nu u_{1,\lambda_\tau^+,\mu}^+ -  \pd_\nu u_{2,\lambda_\tau^+,\mu}^+ , f_\tau^- \rangle_{L^2(\pd \Omega)} = \lim_{\mu \to -\infty} \sum_{n=1}^{+\infty} \left( A_{n,\tau,\mu} + B_{n,\tau,\mu}  \right),\ee
where
\bel{k2}
A_{n,\tau,\mu} := \left( \kappa_{\tau,\mu}(\lambda_{1,n})- \kappa_{\tau,\mu}(\lambda_{2,n}) \right) 
\zeta_{\tau}(\psi_{1,n}),
\ee
and
\bel{k3}
B_{n,\tau,\mu} := \kappa_{\tau,\mu}(\lambda_{2,n}) \left( \zeta_{\tau}(\psi_{1,n}-\psi_{2,n},\psi_{1,n}) +  
\zeta_{\tau}(\psi_{2,n}, \psi_{1,n}-\psi_{2,n}) \right).
\ee
We split the proof into three steps. The first one, presented in Section \ref{sec-an}, is to show that
\bel{k4} 
\lim_{\mu \to -\infty} \sum_{n=1}^{+\infty} A_{n,\tau,\mu} =
\sum_{n=1}^{+\infty} A_{n,\tau,*},\ \mbox{where}\ A_{n,\tau,*}:= \frac{(\lambda_{1,n}-\lambda_{2,n})\zeta_{\tau}(\psi_{1,n})}{(\lambda_\tau^+-\lambda_{1,n})(\lambda_\tau^+-\lambda_{2,n})},
\ee
while the second one, given in Section \ref{sec-bn}, establishes that
\bel{k5} 
\lim_{\mu \to -\infty} \sum_{n=1}^{+\infty} B_{n,\tau,\mu} = 
\sum_{n=1}^{+\infty} B_{n,\tau,*},\ \mbox{where}\
B_{n,\tau,*}:= \frac{\zeta_{\tau}(\psi_{1,n}-\psi_{2,n},\psi_{1,n})+\zeta_{\tau}(\psi_{2,n},\psi_{1,n}-\psi_{2,n})}{\lambda_\tau^+-\lambda_{2,n}}.
\ee
Finally, the end f the proof is displayed in Section \ref{sec-3s}.

\subsubsection{Step 1: Proof of \eqref{k4}}
\label{sec-an}
Let us start by noticing that
$$
\abs{\kappa_{\tau,\mu}(\lambda_{1,n})- \kappa_{\tau,\mu}(\lambda_{2,n})} \leq 2 \abs{\lambda_{1,n}-\lambda_{2,n}}
\max_{t \in [\lambda_{1,n},\lambda_{2,n}]} \left( \frac{1}{\abs{\lambda_\tau^+-t}^2} + \frac{1}{\abs{\mu-t}^2} \right),\ n \in \N.
$$
This can be seen from the identity
$\kappa_{\tau,\mu}(\lambda_{1,n})- \kappa_{\tau,\mu}(\lambda_{2,n})
=- \int_{\lambda_{1,n}}^{\lambda_{2,n}}  \kappa_{\tau,\mu}'(t) dt$, which yields
$$ \abs{\kappa_{\tau,\mu}(\lambda_{1,n})- \kappa_{\tau,\mu}(\lambda_{2,n})} \leq 
\abs{\lambda_{1,n}-\lambda_{2,n}} \max_{t \in [\lambda_{1,n},\lambda_{2,n}]} 
\left( \frac{\abs{\lambda_\tau^+-\mu}}{\abs{\lambda_\tau^+-t} \abs{\mu-t}^2} + \frac{\abs{\lambda_\tau^+-\mu}}{\abs{\lambda_\tau^+-t}^2 \abs{\mu-t}} \right), $$
and from the basic estimate $\abs{\lambda_\tau^+-\mu} \leq \abs{\lambda_\tau^+-t} + \abs{\mu-t}$, entailing:
\beas
\frac{\abs{\lambda_\tau^+-\mu}}{\abs{\lambda_\tau^+-t} \abs{\mu-t}^2} + \frac{\abs{\lambda_\tau^+-\mu}}{\abs{\lambda_\tau^+-t}^2 \abs{\mu-t}}
& \leq & \frac{1}{\abs{\mu-t}^2} + \frac{2}{\abs{\lambda_\tau^+-t} \abs{\mu-t}} +  \frac{1}{\abs{\lambda_\tau^+-t}^2} \\
& \leq & \frac{2}{\abs{\mu-t}^2}+ \frac{2}{\abs{\lambda_\tau^+-t}^2}.
\eeas
Denote by $\lambda_{*,n}$ a real number between $\lambda_{1,n}$ and $\lambda_{2,n}$, where the maximum of the function $t \mapsto  \abs{\lambda_\tau^+-t}^{-2} + \abs{\mu-t}^{-2}$ is achieved, in such a way that we have
\bel{k6}
\abs{\kappa_{\tau,\mu}(\lambda_{1,n})- \kappa_{\tau,\mu}(\lambda_{2,n})} \leq 2 \abs{\lambda_{1,n}-\lambda_{2,n}}
\left( \frac{1}{\abs{\lambda_\tau^+-\lambda_{*,n}}^2} + \frac{1}{\abs{\mu-\lambda_{*,n}}^2} \right),\ n \in \N.
\ee
Next, bearing in mind that $\lim_{n \to +\infty} \lambda_{1,n}=+\infty$, we pick $N_0 \in \N$ so large, that 
\bel{k7}
\lambda_{1,N_0} \geq \abs{\lambda_\tau^+} + 4 M.
\ee
Since $\lambda_{1,n} \geq \lambda_{1,N_0}$ for all $n \geq N_0$, we have $\abs{\lambda_\tau^+-\lambda_{1,n}} \geq \lambda_{1,n} -\abs{\lambda_\tau^+} \geq 4M$ in this case, whence
$\abs{\lambda_\tau^+-\lambda_{*,n}} \geq \abs{\lambda_\tau^+-\lambda_{1,n}} -  \abs{\lambda_{1,n}-\lambda_{*,n}}
 \geq \abs{\lambda_\tau^+-\lambda_{1,n}} -  2M$. Here, we used the basic inequality $\abs{\lambda_{1,n}-\lambda_{*,n}} \leq \abs{\lambda_{1,n}-\lambda_{2,n}}$ and the estimate 
\bel{k7*}
\abs{\lambda_{1,n}-\lambda_{2,n}} \leq \norm{q_1-q_2}_{L^\infty(\Omega)} \leq 2 M,\ n \in \N,
\ee 
arising from the Min-Max principle
and the operator identity $A_{q_2}=A_{q_1}+q_2-q_1$. Therefore, we have
\bel{k8*}
\abs{\lambda_\tau^+-\lambda_{*,n}} \geq \frac{\abs{\lambda_\tau^+-\lambda_{1,n}}}{2},\ n \geq N_0.
\ee
Similarly, taking $\mu \in \left(-\infty, -(1+5M) \right)$, we have $\abs{\mu-\lambda_{1,n}} \geq -\mu - M \geq 4M$. Since $\abs{\lambda_{1,n}-\lambda_{*,n}} \leq 2M$, by \eqref{k7*}, we get that $\abs{\lambda_{1,n}-\lambda_{*,n}} \leq \abs{\mu-\lambda_{1,n}} \slash 2$, and hence
$$ \abs{\mu-\lambda_{*,n}} \geq \abs{\mu-\lambda_{1,n}}  - \abs{\lambda_{1,n} - \lambda_{*,n}} \geq \frac{\abs{\mu-\lambda_{1,n}}}{2},\ n \in \N. $$
Putting this together with \eqref{k2},  \eqref{k6} and \eqref{k8*}, we obtain that
\bel{k8}
\abs{A_{n,\tau,\mu}} \leq 8 \delta_1 \left( \frac{\abs{\zeta_{\tau}(\psi_{1,n})}}{\abs{\lambda_\tau^+-\lambda_{1,n}}^2} + \frac{\abs{\zeta_{\tau}(\psi_{1,n})}}{\abs{\mu-\lambda_{1,n}}^2} \right),\ n \geq N_0,
\ee
where $\delta_1 := \sup_{n \in \N} \abs{\lambda_{1,n}-\lambda_{2,n}} < \infty$. 

Further, in light of \eqref{p5}, we deduce from \eqref{f3} that
\beas
\sum_{n =1}^{+\infty} \frac{\abs{\zeta_{\tau}(\psi_{1,n})}}{\abs{\ell-\lambda_{1,n}}^2} & \leq &
\left( \sum_{n =1}^{+\infty} \abs{\frac{\langle f_\tau^+ , \psi_{1,n} \rangle_{L^2(\pd \Omega)}}{\ell-\lambda_{1,n}}}^2 \right)^{1 \slash 2}
\left( \sum_{n =1}^{+\infty} \abs{\frac{\langle f_\tau^- , \psi_{1,n} \rangle_{L^2(\pd \Omega)}}{\overline{\ell}-\lambda_{1,n}}}^2 \right)^{1 \slash 2} \\
& \leq & \norm{u_{1,\ell}^+}_{L^2(\Omega)}  
\norm{u_{1,\overline{\ell}}^-}_{L^2(\Omega)},\ \ell=\lambda_\tau^+,\ \mu.
\eeas
Here, $u_{1,\ell}^\pm$ denotes the $H^2(\Omega)$-solution to \eqref{f1} where $(\ell,q_1,f_\tau^\pm)$ is substituted for $(\lambda,q,f)$.
Thus, bearing in mind that $\overline{\lambda_\tau^+}=\lambda_\tau^-$, we derive from \eqref{k8} that
$$  \sum_{n =N_0}^{+\infty} \abs{A_{n,\tau,\mu}} \leq 8 \delta_1 \left( \norm{u_{1,\lambda_\tau^+}^+}_{L^2(\Omega)}  
\norm{u_{1,\lambda_\tau^-}^-}_{L^2(\Omega)} + \norm{u_{1,\mu}^+}_{L^2(\Omega)}  
\norm{u_{1,\mu}^-}_{L^2(\Omega)} \right). $$
With reference to \eqref{f3}, we assume upon possibly enlarging $-\mu$, that $\norm{u_{1,\mu}^\pm}_{L^2(\Omega)} \leq 1$, and obtain
\bel{k8a}  
\sum_{n =N_0}^{+\infty} \abs{A_{n,\tau,\mu}} \leq 2 \delta_1 \left( (M+2)^2 c_*^2  + 4 \right),
\ee
with the aid of \eqref{g6e}. Now, since $\lim_{\mu \to -\infty} A_{n,\tau,\mu}=A_{n,\tau,*}$ for all $n \in \N$, we deduce
\eqref{k4} from this and \eqref{k8a} by invoking the Lebesgue dominated convergence theorem.


\subsubsection{Step 2: Proof of \eqref{k5}}
\label{sec-bn}
For all $n \geq N_0$, we infer from \eqref{k7}-\eqref{k7*} that
$$ \lambda_{2,n} \geq \lambda_{1,n} - \abs{\lambda_{1,n}-\lambda_{2,n}} \geq \lambda_{1,N_0} - 2M \geq \abs{\lambda_\tau^+}, $$
whence $\abs{\mu-\lambda_{2,n}} = \lambda_{2,n} - \mu \geq \abs{\lambda_\tau^+} - \mu \geq \abs{\lambda_\tau^+ - \mu}$.
Therefore, we get
$$ \abs{\kappa_{\tau,\mu}(\lambda_{2,n})} \leq \frac{1}{\abs{\lambda_\tau^+-\lambda_{2,n}}},\ n \geq N_0, $$ 
from \eqref{p4}. This and \eqref{p5} entail
\bea
\abs{B_{n,\tau,\mu}} & \leq & \norm{\psi_{1,n} - \psi_{2,n}}_{L^2(\pd \Omega)} \left( 
\norm{f_\tau^+}_{L^2(\pd \Omega)} \abs{\frac{ \langle f_\tau^- , \psi_{1,n} \rangle_{L^2(\pd \Omega)}}{\lambda_\tau^+-\lambda_{2,n}}} + \norm{f_\tau^-}_{L^2(\pd \Omega)}
\abs{\frac{ \langle f_\tau^+ , \psi_{2,n} \rangle_{L^2(\pd \Omega)}}{\lambda_\tau^+-\lambda_{2,n}}} \right) \nonumber \\
& \leq &  c_* \norm{\psi_{1,n} - \psi_{2,n}}_{L^2(\pd \Omega)} \left( 
\abs{\frac{ \langle f_\tau^- , \psi_{1,n} \rangle_{L^2(\pd \Omega)}}{\lambda_\tau^+-\lambda_{2,n}}} + 
\abs{\frac{ \langle f_\tau^+ , \psi_{2,n} \rangle_{L^2(\pd \Omega)}}{\lambda_\tau^+-\lambda_{2,n}}} \right),
\label{k9}
\eea
by applying \eqref{g*} with $(p,X)=(2,\pd \Omega)$. 

Further, recalling from \eqref{k7} that $\abs{\lambda_\tau^+-\lambda_{1,n}} \geq 2 \norm{q_1-q_2}_{L^\infty(\Omega)}$ for all $n \geq N_0$, and using the estimate $\abs{\lambda_\tau^+-\lambda_{2,n}} \geq \abs{\lambda_\tau^+-\lambda_{1,n}} -  \norm{q_1-q_2}_{L^\infty(\Omega)}$ arising from \eqref{k7*}, we find that
\bel{k9*}
\abs{\lambda_\tau^+-\lambda_{2,n}} \geq \frac{\abs{\lambda_\tau^+-\lambda_{1,n}}}{2}=\frac{\abs{\lambda_\tau^- - \lambda_{1,n}}}{2},\ n \geq N_0.
\ee
Putting this together with \eqref{k9}, we get for every $\mu \in \left(-\infty,-(1+5M) \right]$, that
$$
\sum_{n=N_0}^{+\infty} \abs{B_{n,\tau,\mu}} \leq  2 c_* \eps_1 \left( 
\left( \sum_{n=N_0}^{+\infty} \abs{\frac{\langle f_\tau^- , \psi_{1,n} \rangle_{L^2(\pd \Omega)}}{\lambda_\tau^- -\lambda_{1,n}}}^2 \right)^{1 \slash 2} + 
\left( \sum_{n=N_0}^{+\infty} \abs{\frac{ \langle f_\tau^+ , \psi_{2,n} \rangle_{L^2(\pd \Omega)}}{\lambda_\tau^+-\lambda_{2,n}}}^2 \right)^{1 \slash 2} \right),
$$
with $\eps_1:= \left( \sum_{n=1}^{+\infty} \norm{\psi_{1,n} - \psi_{2,n}}_{L^2(\pd \Omega)}^2 \right)^{1 \slash 2}$. This leads to
$$\sum_{n=N_0}^{+\infty} \abs{B_{n,\tau,\mu}} \leq  2 c_* \eps_1 \left( 
\norm{u_{1,\lambda_\tau^-}^-}_{L^2(\Omega)}  + 
\norm{u_{2,\lambda_\tau^+}^+}_{L^2(\Omega)}  \right) \leq 8 (M+2) c_*^2 ,
$$
with the aid of Lemma \ref{lm1} and \eqref{g6e}. Now, \eqref{k5} follows from this and the identities
$$\lim_{\mu \to -\infty} B_{n,\tau,\mu}=B_{n,\tau,*},\ n \in \N, $$
by applying Lebesgue's dominated convergence theorem. \\


\subsubsection{Step 3: End of the proof}
\label{sec-3s}

Putting \eqref{k1} and \eqref{k4}-\eqref{k5} together, we obtain that
\bel{k10}
S_\tau = \sum_{n=1}^{+\infty} \left( A_{n,\tau,*} + B_{n,\tau,*} \right).
\ee
Further, since $\Pim{\lambda_\tau^+-\lambda_{j,n}}=2 \tau$ for $j=1,2$ and all $n \in \N$, by \eqref{gg2}, we infer from \eqref{g*} with $(p,X)=(2, \pd \Omega)$, \eqref{p5},
and \eqref{k4}-\eqref{k5} that
$$ \abs{A_{n,\tau,*}} \leq c_*^2 \frac{\abs{\lambda_{1,n}-\lambda_{2,n}} \norm{\psi_{1,n}}_{L^2(\pd \Omega)}^2}{\tau^2} $$ 
and
$$
\abs{B_{n,\tau,*}} \leq c_*^2 \frac{\left( \norm{\psi_{1,n}}_{L^2(\pd \Omega)} + \norm{\psi_{2,n}}_{L^2(\pd \Omega)} \right) \norm{\psi_{1,n}-\psi_{2,n}}_{L^2(\pd \Omega)}}{\tau}. $$
Therefore, it holds true for all $n \in \N$ that $\lim_{\tau \to +\infty} A_{n,\tau,*} = \lim_{\tau \to +\infty} B_{n,\tau,*} = 0$, so it follows from \eqref{k10} that 
\bel{k10*}
 \limsup_{\tau \to +\infty} \abs{S_\tau} \leq \limsup_{\tau \to +\infty} \sum_{n=N}^{+\infty}  \abs{A_{n,\tau,*}}
+ \limsup_{\tau \to +\infty} \sum_{n=N}^{+\infty}  \abs{B_{n,\tau,*}},\ N \in \N.
\ee
Moreover, setting $\delta_N :=  \sup_{n \geq N} \abs{\lambda_{1,n}-\lambda_{2,n}}$, we infer from \eqref{p5} and \eqref{k4} that 
\beas
\sum_{n=N}^{+\infty}  \abs{A_{n,\tau,*}} & \leq & \delta_N
\sum_{n=N}^{+\infty} \abs{\frac{\langle f_\tau^- , \psi_{1,n} \rangle_{L^2(\pd \Omega)}}{\lambda_\tau^+-\lambda_{1,n}}}
\abs{\frac{\langle f_\tau^+ , \psi_{1,n} \rangle_{L^2(\pd \Omega)}}{\lambda_\tau^+-\lambda_{2,n}}} \\
& \leq &  2 \delta_N 
\left( \sum_{n=N}^{+\infty} \abs{\frac{\langle f_\tau^- , \psi_{1,n} \rangle_{L^2(\pd \Omega)}}{\lambda_\tau^- - \lambda_{1,n}}}^2 \right)^{1 \slash 2}
\left( \sum_{n=N}^{+\infty} \abs{\frac{\langle f_\tau^+ , \psi_{1,n} \rangle_{L^2(\pd \Omega)}}{\lambda_\tau^+-\lambda_{1,n}}}^2 \right)^{1 \slash 2},\ N \geq N_0.
\eeas
In the last line we used the Cauchy-Schwarz inequality, the estimate \eqref{k9*} and the identity $\abs{\lambda_\tau^- - \lambda_{1,n}}=\abs{\lambda_\tau^+ - \lambda_{1,n}}$. Therefore, applying Lemma \ref{lm1} and \eqref{g6e}, we obtain for all $N \geq N_0$, that
$$ \sum_{n=N}^{+\infty}  \abs{A_{n,\tau,*}}
\leq  2 \delta_N  \norm{u_{1,\lambda_\tau^-}^-}_{L^2(\pd \Omega)}  \norm{u_{1,\lambda_\tau^+}^+}_{L^2(\pd \Omega)}
\leq  \frac{(M+2)^2}{2} c_*^2 \delta_N, $$ 
which entails
\bel{k11} 
\limsup_{\tau \to +\infty} \sum_{n=N}^{+\infty} \abs{A_{n,\tau,*}} \leq  \frac{(M+2)^2}{2} c_*^2 \left( \sup_{n \geq N} \abs{\lambda_{1,n}-\lambda_{2,n}} \right),\ N \geq N_0.
\ee
Similarly, using \eqref{p5} and \eqref{k5}, we can upper bound $\sum_{n=N}^{+\infty}  \abs{B_{n,\tau,*}}$ by 
\beas
& &  \sum_{n=N}^{+\infty} \norm{\psi_{1,n} -  \psi_{2,n}}_{L^2(\pd \Omega)}  \left(  
\norm{f_\tau^+}_{L^2(\pd \Omega)} \abs{\frac{\langle f_\tau^- , \psi_{1,n} \rangle_{L^2(\pd \Omega)}}{\lambda_\tau^+-\lambda_{2,n}}}
+
\norm{f_\tau^-}_{L^2(\pd \Omega)} \abs{\frac{\langle f_\tau^+ , \psi_{2,n} \rangle_{L^2(\pd \Omega)}}{\lambda_\tau^+-\lambda_{2,n}}} \right) \nonumber \\
& \leq & \eps_N
\left(
\norm{f_\tau^+}_{L^2(\pd \Omega)} \left(  \sum_{n=N}^{+\infty} \abs{\frac{\langle f_\tau^- , \psi_{1,n} \rangle_{L^2(\pd \Omega)}}{\lambda_\tau^+-\lambda_{2,n}}}^2  \right)^{1 \slash 2}
+
\norm{f_\tau^-}_{L^2(\pd \Omega)} \left(  \sum_{n=N}^{+\infty} \abs{\frac{\langle f_\tau^+ , \psi_{2,n} \rangle_{L^2(\pd \Omega)}}{\lambda_\tau^+-\lambda_{2,n}}}^2  \right)^{1 \slash 2}
\right), \label{k14}
\eeas
where $\eps_N:=  \left(  \sum_{n=N}^{+\infty}   \norm{\psi_{1,n} -  \psi_{2,n}}_{L^2(\pd \Omega)}^2 \right)^{1 \slash 2}$. Next, applying Lemma \ref{lm1} we get for all $N \in \N$ that
$\sum_{n=N}^{+\infty} \abs{\frac{\langle f_\tau^+ , \psi_{2,n} \rangle_{L^2(\pd \Omega)}}{\lambda_\tau^+-\lambda_{2,n}}}^2 \leq \norm{u_{2,\lambda_\tau^+}^+}_{L^2(\Omega)}^2$, and for every $N \geq N_0$ that
$$ \sum_{n=N}^{+\infty} \abs{\frac{\langle f_\tau^- , \psi_{1,n} \rangle_{L^2(\pd \Omega)}}{\lambda_\tau^+-\lambda_{2,n}}}^2 \leq
4 \sum_{n=N}^{+\infty} \abs{\frac{\langle f_\tau^- , \psi_{1,n} \rangle_{L^2(\pd \Omega)}}{\lambda_\tau^+-\lambda_{1,n}}}^2 \leq 
4 \norm{u_{1,\lambda_\tau^-}^-}_{L^2(\Omega)}^2,
$$
by virtue of \eqref{k9*}. Therefore, in light of \eqref{g*} with $(p,X)=(2,\pd \Omega)$ and \eqref{g6e}, we have
$$
\sum_{n=N}^{+\infty}  \abs{B_{n,\tau,*}} \leq \eps_N \left( 2 \norm{f_\tau^+}_{L^2(\pd \Omega)} \norm{u_{1,\lambda_\tau^-}^-}_{L^2(\Omega)} +
\norm{f_\tau^-}_{L^2(\pd \Omega)} \norm{u_{2,\lambda_\tau^+}^+}_{L^2(\Omega)} 
\right)  \leq \frac{3(M+2)}{2} c_*^2 \eps_N,
$$
provided $N \geq N_0$, and hence
$$
\limsup_{\tau \to +\infty} \sum_{n=N}^{+\infty}  \abs{B_{n,\lambda_\tau^+,*}} \leq  \frac{3(M+2)}{2} c_*^2 \left(  \sum_{n=N}^{+\infty}   \norm{\psi_{1,n} -  \psi_{2,n}}_{L^2(\pd \Omega)}^2 \right)^{1 \slash 2},\ N \geq N_0.
$$
Putting this together with \eqref{k10*}-\eqref{k11}, we obtain
\bel{k12}  
\limsup_{\tau \to +\infty} \abs{S_\tau} \leq c \left( \left( \sum_{n=N}^{+\infty} \norm{\psi_{1,n} -  \psi_{2,n}}_{L^2(\pd \Omega)}^2 \right)^{1 \slash 2} + \sup_{n \geq N} \abs{\lambda_{1,n}-\lambda_{2,n}} \right),\ N \geq N_0,
\ee
where the constant $c:= (M+2) (M +5) c_*^2  \slash 2$ is independent of $N$. Now, by sending $N$ to $+\infty$ in the right hand side of the above estimate, we get that $\limsup_{\tau \to +\infty} \abs{S_\tau}=0$. Thus, we have $\lim_{\tau \to +\infty} S_\tau=0$, by virtue of Proposition \ref{pr-iso}. This entails in the same way as in Section \ref{sec-prthm1} that $q_1=q_2$ in $\Omega$, which terminates the proof of Theorem \ref{thm-BLmd2}


\subsection{The stability issue}
\label{sec-stability}

The stability issue for the Borg-Levinson inverse problem was first examined by G. Alessandrini and J. Sylvester in \cite{AS}, who proved H\"older stable determination of $q$ by $\bsd(q)$ (see also \cite[Theorem 2.31]{C} for a reformulation of their result). We shall establish in this section, at the expense of stronger regularity on $q$, that it can be H\"older-stably determined by the asymptotic behavior of its BSD, provided $q$ is known on the boundary $\pd \Omega$. 

\subsubsection{Notations and stability inequality}

We stick with the notations of Section \ref{sec-BLth}.
In particular, given two real-valued potentials $q_j$, $j=1,2$, we denote by $\{ \lambda_{j,n},\ n \in \N \}$ the sequence of the eigenvalues of $A_{q_j}$, arranged in non-decreasing order (and repeated with the multiplicity), and we write $\psi_{j,n}$ instead of $\pd_\nu \varphi_{j,n}$ for all $n \in \N$, where $\{ \varphi_{j,n},\ n \in \N \}$ is a $L^2(\Omega)$-orthonormal basis of eigenvectors of $A_{q_j}$, such that $A_{q_j} \varphi_{j,n} =  \lambda_{j,n} \varphi_{j,n}$. 

\begin{theorem}
\label{thm-BLmd3}
For $M \in (0,+\infty)$ fixed, pick $q_1$ and $q_2$ in $L^\infty(\Omega,\R) \cap H^1(\Omega)$, such that
\bel{s1}
\norm{q_j}_{L^\infty(\Omega)} +  \norm{q_j}_{H^1(\Omega)} \leq M,\ j=1,2,
\ee
and
\bel{s2}
q_1 = q_2\ \mbox{on}\ \pd \Omega.
\ee
Assume moreover that 
\bel{s3}
\sum_{n=1}^{+\infty} \norm{\psi_{1,n}-\psi_{2,n}}_{L^2(\pd \Omega)}^2 < \infty.
\ee
Then, the following stability estimate
$$ 
\norm{q_1-q_2}_{L^2(\Omega)} \leq C \limsup_{n \to +\infty} \abs{\lambda_{1,n}-\lambda_{2,n}}^{2 \slash (d+2)},
$$
holds for some positive constant $C$ that depends only on $\Omega$ and $M$.
\end{theorem}

\subsubsection{Proof of Theorem \ref{thm-BLmd3}}
Let us recall from \eqref{k12} that for all $N \geq N_0$, we have
$$
\limsup_{\tau \to +\infty} \abs{S_\tau} \leq c \left( \left( \sum_{n=N}^{+\infty} \norm{\psi_{1,n} -  \psi_{2,n}}_{L^2(\pd \Omega)}^2 \right)^{1 \slash 2} + \sup_{n \geq N} \abs{\lambda_{1,n}-\lambda_{2,n}} \right),
$$
for some positive constant $c$ that is independent of $N$ and $\xi$. Thus, in light of \eqref{s3} we get upon sending $N$ to infinity, that
\bel{s5} 
\limsup_{\tau \to +\infty} \abs{S_\tau} \leq c \limsup_{n \to \infty} \abs{\lambda_{1,n}-\lambda_{2,n}}.
\ee
Further, we recall from Proposition \ref{pr-iso} that 
$$ \lim_{\tau \to +\infty} S_\tau = \int_{\Omega} q(x) e^{-x \cdot \xi} dx = (2 \pi)^{d \slash 2} \widehat{q}(\xi),$$
where $q$ is the same as in \eqref{p*} and $\widehat{q}$ stands for the Fourier transform $\cF q$ of $q$, defined by \eqref{p0}.
This, \eqref{s5} and the basic estimate $\abs{\lim_{\tau \to +\infty} S_\tau} \leq \limsup_{\tau \to +\infty} \abs{S_\tau}$, yield
$\abs{\widehat{q}(\xi)} \leq (2 \pi)^{-d \slash 2} c \limsup_{n \to +\infty}  \abs{\lambda_{1,n}-\lambda_{2,n}}$, uniformly in $\xi \in \R^d$ . Thus, we obtain
\bel{s6}
\norm{\widehat{q}}_{L^\infty(\R^d)} \leq c \limsup_{n \to +\infty}  \abs{\lambda_{1,n}-\lambda_{2,n}}.
\ee
upon substituting $(2 \pi)^{-d \slash 2} c$ for $c$. 

On the other hand, we infer from \eqref{p*} and the Plancherel theorem that
\bel{s7} 
\norm{q_1-q_2}_{L^2(\Omega)}^2 = \norm{q}_{L^2(\R^d)}^2 = \norm{\widehat{q}}_{L^2(\R^d)}^2 =\int_{\R^d} \abs{\widehat{q}(\xi)}^2 d \xi. \ee
For $R \in (1,+\infty)$ fixed, set $B_R:= \{ \xi \in \R^d,\ \abs{\xi} \leq R\}$ and notice from \eqref{s7} that
\bel{s8}
\norm{q_1-q_2}_{L^2(\Omega)}^2  =  \int_{B_R} \abs{\widehat{q}(\xi)}^2 d \xi + \int_{\R^d \setminus B_R} \abs{\widehat{q}(\xi)}^2 d \xi. 
\ee
The first term in the right hand side of \eqref{s8} is easily treated, as we have
\bel{s8b}  
\int_{B_R} \abs{\widehat{q}(\xi)}^2 d \xi \leq \tilde{c} R^d \norm{\widehat{q}}_{L^\infty(B_R)}^2,
\ee
for some positive constant $\tilde{c}$ that is independent of $R$. Further, since $q_1-q_2 \in H_0^1(\Omega)$ from \eqref{s2}, we see that  $q \in H^1(\R^d)$. Thus we may write
$$ \int_{\R^d \setminus B_R} (1+\abs{\xi}^2) \abs{\widehat{q}(\xi)}^2 d \xi= \norm{q}_{H^1(\R^d)}^2\\
= \norm{q_1-q_2}_{H^1(\Omega)}^2 \leq  \left( \norm{q_1}_{H^1(\Omega)} + \norm{q_2}_{H^1(\Omega)} \right)^2 \leq 4M^2, $$
from \eqref{s1}, and consequently
$$ \int_{\R^d \setminus B_R} \abs{\widehat{q}(\xi)}^2 d \xi \leq R^{-2} \int_{\R^d \setminus B_R} (1+\abs{\xi}^2) \abs{\widehat{q}(\xi)}^2 d \xi \leq 4 M^2 R^{-2}. $$
Putting this and \eqref{s8}-\eqref{s8b} together, we find that
\bel{s9}
\norm{q_1-q_2}_{L^2(\Omega)}^2  \leq \tilde{c} \left( R^d  \norm{\widehat{q}}_{L^\infty(B_R)}^2 + R^{-2} \right),
\ee
upon possibly substituting $\max(\tilde{c},4 M^2)$ for $\tilde{c}$.

Set $\delta:=\limsup_{n \to +\infty} \abs{\lambda_{1,n}-\lambda_{2,n}}$. We shall examine the two cases $\delta \in (0,1)$ and $\delta \in [1,+\infty)$ separately.
In the first case we plug the estimate $\norm{\widehat{q}}_{L^\infty(B_R)} \leq c \delta$, arising from \eqref{s6}, in \eqref{s9}, choose $R=\delta^{-2 \slash (d+2)}$, and get 
\bel{s10}
\norm{q_1-q_2}_{L^2(\Omega)}  \leq C \delta^{2 \slash (d+2)},
\ee
with  $C:=\left( \tilde{c}(1+c^2) \right)^{1 \slash 2}$.
In the second case we have obviously
$$
\norm{q_1-q_2}_{L^2(\Omega)}  \leq \norm{q_1}_{L^2(\Omega)} + \norm{q_2}_{L^2(\Omega)} \leq 2 M \leq 2M \delta^{2 \slash (d+2)},\ \delta \in [1,+\infty),
$$
so the desired result follows from this and \eqref{s10}.


\section{Application to parabolic inverse coefficient problems}
\label{sec-appl:par}
\setcounter{equation}{0}



Let $T \in (0,+\infty)$, let $\Omega$ be as in the preceding sections, that is $\Omega \subset \R^d$, $d \geq 2$, is a bounded domain with boundary $\pd \Omega \in C^{1,1}$.
We consider the diffusion equation
\bel{a1}
\left\{ \begin{array}{rcll}
(\pd_t - \Delta + q) u & = & 0 & \mbox{in}\ Q:= \Omega \times (0,T)  \\
u & = & f & \mbox{on}\ \Sigma := \pd \Omega \times (0,T) \\
u(\cdot,0) & = & 0 & \mbox{in}\ \Omega, \end{array} \right.
\ee
where $q$ is a real-valued bounded potential and $f$ fulfills the compatibility condition: 
$$f(\cdot,0)=0\ \mbox{on}\ \pd \Omega.$$

The inverse problem we examine in this section can be stated as follows. Given $M \in (0,+\infty)$ and two open subsets $\Gi$ and $\Go$ of $\pd \Omega$, determine 
$$ q \in Q_M := \{ q \in L^\infty(\Omega,\R),\ \norm{q}_{L^\infty(\Omega)} \leq M \} $$
by knowledge of the parabolic partial DN map at one fixed time $T_0 \in (0,T)$:
$$ \Lambda_q : f \in \sHi \mapsto \partial_\nu(\cdot,T_0)_{| \Go}. $$
Here, we have set $\sHi := \{ f \in \sH,\ \supp f \subset  \Gi \times (0,T_0) \}$, with $\sH:= C^{1,\alpha}([0,T],H^{3 \slash 2}(\pd \Omega))$ for some $\alpha \in (0,1]$. 

\begin{remark} Since $\supp f \subset (0,T_0) \times \Gi$ for any $f \in \sHi$, then the compatibility condition $f(\cdot,0)=0$ holds on $\pd \Omega$.
\end{remark}

Evidently, the inverse problem under investigation can be reformulated as whether the mapping $q \in Q_M \mapsto \Lambda_q$ is injective.


\subsection{Parabolic Dirichlet-to-Neumann map and identifiability}
\label{sec-DN}

We start by recalling the following uniqueness and existence result (see e.g. \cite[Section 3.5]{C}).
\begin{proposition}
\label{pr-eu}
For all $q \in Q_M$ and all $f \in \sH$, there exists a unique solution
$$u \in \sZ := C^1((0,T],H^2(\Omega)) \cap C^0([0,T],L^2(\Omega)) $$ 
to \eqref{a1}. 
\end{proposition}
Thus, by continuity of the trace operator $g \to (\partial_{\nu} g)_{| \Go}$ from $H^2(\Omega)$ to $H^{1 \slash 2}(\Go)$, the map
$$ \begin{array}{cccc}
\Lambda_q : & \sHi & \to & H^{1\slash 2}(\Go) \\ & f & \mapsto & \pd_\nu u(\cdot,T_0)_{| \Go} 
\end{array} $$
is well-defined. 

The main result if this section is as follows.
\begin{theorem}
\label{thm-main}
Assume that $\Gi \cup \Go=\pd \Omega$ and that $\Gi \cap \Go \neq \emptyset$.
For $j=1,2$, let $q_j \in Q_M$ and put $\Lambda_j:=\Lambda_{q_j}$. Then, we have the implication:
$$ \left( \forall f \in \sHi,\ \Lambda_1(f)=\Lambda_2(f) \right) \Longrightarrow (q_1=q_2). $$
\end{theorem}

\begin{remark}
This result was proved by B. Canuto and O. Kavian in \cite{CK}. Recently, in \cite{KOSY}, it was extended  to the case of time-fractional diffusion equations $(\pd_t^\alpha - \Delta + q) u = 0$ in $Q$, with $\alpha \in (0,1) \cup (1,2)$, and where $\pd_t^\alpha$ denotes the Caputo fractional derivative of order $\alpha$.
\end{remark}


\subsection{Technical tools}

The proof of Theorem \ref{thm-main} consists of two steps. 
\begin{itemize}
\item The first one is to show that knowledge of $\Lambda_q$ uniquely determines\footnote{Or, equivalently, the BSD associated with the Dirichlet Laplacian $A_q$, defined by \eqref{i0}.} $\bsd(q)$ :
\bel{ep0}
\left( \forall f \in \sHi,\ \Lambda_1(f) = \Lambda_2(f) \right) \Longrightarrow \left( \bsd(q_1) = \bsd(q_2) \right).
\ee
\item The second step is to identify $q$ through $\bsd(q)$, with the aid of Theorem \ref{thm-BLmd1}.
\end{itemize}

In \cite{KKLM}, A. Katchalov, Y. Kurylev, M. Lassas and C. Mandache have established the equivalence between the full parabolic DN map and the BSD. Their statement is quite similar to the claim of the first step, except that this is the partial data 
$\Lambda_q$ (and not the full parabolic DN map) that is considered here and that we only seek determination of $\bsd(q)$ by $\Lambda_q$ (and not equivalence of these two data).

\subsubsection{Some notations and useful properties}
We stick with the notations of Section \ref{sec-BLth}. That is to say that for $j=1,2$, 
we write $\bsd(q_j) = \{(\lambda_{j,n},\psi_{j,n}),\ n \in \N \}$ with $\psi_{j,n}=\pd_\nu \varphi_{j,n}$.\\

\paragraph{\it Weyl's law.} It is well known (see e.g. \cite[Section XIII.15]{RS4}) that there exist two constants $n_M \in \N$ and $c \in (1,+\infty)$, both of them depending only on $\Omega$ and $M$, such that we have
$$c^{-1} n^{2 \slash d} \leq \lambda_n \leq c n^{2 \slash d},\ n \geq n_M. $$
This entails for all $k\in \N$ and all $\eps >0$, that the series
\bel{p2}
\sum_{n=n_M}^{+\infty} \lambda_n^k e^{-\eps \lambda_n} \leq c^k \sum_{n=n_M}^{+\infty} n^{2 k \slash d} e^{-c^{-1} \eps n^{2 \slash d}} < \infty.
\ee

\paragraph{\it  Linear independence of the Neumann data.}
Given a non-empty open subset $\Gamma$ of $\pd \Omega$, the family $\{ {\psi_n}_{| \Gamma},\ n \in \N \}$ is, in general, not linearly independent in $L^2(\Gamma)$, but the normal derivatives of the eigenfunctions associated with one eigenvalue are linearly independent. More precisely, if $m_n$ denotes the geometric multiplicity of $\lambda_n$, let $\{ \varphi_{n,i},\ i=1,\ldots,m_n \}$ be an orthonormal basis of the $L^2(\Omega)$-subspace $\ker(A_q-\lambda_n)$. Then, we have
\bel{p3}
\dim \{  {\psi_{n,i}}_{| \Gamma},\ i=1,\ldots,m_n \} = m_n. 
 \ee
The proof of \eqref{p3} essentially relies on the following unique continuation principle for local Cauchy data.

\begin{lemma}
\label{lm-ucp}
For $j,k=1,\ldots,d$, let $a_{j,k}=a_{k,j} \in W^{1,\infty}(\Omega)$, $b_j \in L^\infty(\Omega)$ and $c \in L^\infty(\Omega)$, and suppose that the differential operator
$$ P := -\sum_{j,k=1}^d a_{j,k}(x) \pd_{j,k}^2 + \sum_{j=1}^d b_j(x) \pd_j + c(x), $$
fulfills the ellipticity condition:
$$ \exists \lambda \in (0,+\infty),\ \sum_{j,k} a_{j,k}(x) \xi_j \xi_k \geq \lambda \abs{\xi}^2,\ x \in \Omega,\ \xi \in \R^d. $$
Then, for all $u \in H^2(\Omega)$, we have:
$$ ( Pu =0,\ u_{| \Gamma} = \pd_{\nu} u_{| \Gamma} = 0) \Longrightarrow ( u=0\ \mbox{in}\ \Omega). $$
\end{lemma}

To show \eqref{p3}, we pick $m_n$ constants $c_i$, $i=1,\ldots,m_n$, such that 
$$ \sum_{i=1}^{m_n} c_i \psi_{n,i} = 0\ \mbox{on}\ \Gamma. $$
Putting $\varphi:=\sum_{i=1}^{m_n} c_i \varphi_{n,i}$, we see that $\varphi$ solves
$$ \left\{ \begin{array}{rcll} A_q \varphi & = & \lambda_n \varphi & \mbox{in}\  \Omega \\
\varphi & = & 0 & \mbox{on}\ \Gamma \\
\pd_\nu \varphi & = & \sum_{i=1}^{m_n} c_i \psi_{n,i}  = 0 & \mbox{on}\ \Gamma. \end{array} \right. $$
Thus, upon applying Lemma \ref{lm-ucp} with $a_{jk}=\delta_{jk}$ and $b_j=0$ for $j,k=1,\ldots,d$, and $c=q$, we get tat $\varphi=0$ in $\Omega$. Here and below, $\delta$ denotes the Kronecker symbol, i.e.
$$ \delta_{jk}:= \left\{ \begin{array}{ll} 1 & \mbox{if}\ j=k \\ 0 & \mbox{otherwise}. \end{array} \right.$$
Next, since  the family
$\{ \varphi_{n,i},\ i=1,\ldots,m_n \}$ is  linearly independent in $L^2(\Omega)$, we obtain that $c_i=0$ for all $i=1,\ldots,m_n$. Therefore, the ${\psi_{n,i}}_{| \Gamma}$, $i=1,\ldots,m_n$, are linearly independent in $L^2(\Gamma)$, which establishes \eqref{p3}.\\

We may now prove the:
\begin{lemma}
\label{lm-nz}
Let $\Gamma$ and $\Gamma'$ be two non-empty open subsets of $\pd \Omega$. Then, the function
$$ \theta_n(\sigma,\sigma')= \sum_{i=1}^{m_n} \psi_{n,i}(\sigma) \psi_{n,i}(\sigma'),\ (\sigma,\sigma') \in \Gamma \times \Gamma', $$
is not identically zero in $\Gamma \times \Gamma'$.
\end{lemma}
\begin{proof}
We prove Lemma \ref{lm-nz} by contradiction. If we assume that $\theta_n(\sigma,\sigma')=0$ for a.e. $(\sigma,\sigma') \in \Gamma \times \Gamma'$, then we would have
$$ \sum_{i=1}^{m_n} \psi_{n,i}(\sigma) \psi_{n,i} = 0\ \mbox{on}\ \Gamma', $$
and hence
$$\psi_{n,i}=0\ \mbox{on}\ \Gamma,\ i=1,\ldots,m_n, $$
from \eqref{p3}, which is in contradiction with \eqref{p3}. Therefore, $\theta_n$ is not identically zero in $\Gamma \times \Gamma'$.
\end{proof}


\subsection{Proof of the parabolic uniqueness result}
\label{sec-proof}

In this section we prove Theorem \ref{thm-main}. In light of Theorem \ref{thm-BLmd1}, it is enough to show \eqref{ep0}.
More precisely, we shall establish the:

\begin{theorem}
\label{thm-eq}
Under the assumptions of Theorem \ref{thm-main}, we have
$$ (\forall f \in \sHi,\ \Lambda_1(f) = \Lambda_2(f) )  \Longrightarrow ( \lambda_{1,n}=\lambda_{2,n}\ \mbox{and}\ \psi_{1,n}=\psi_{2,n}\ \mbox{on}\ \Gi \cup \Go,\ n \in \N) ,$$
up to an appropriate choice of the eigenfunctions $\varphi_{2,n}$ of $A_2$. 
\end{theorem}

Prior to proving Theorem \ref{thm-eq}, we establish a representation formula of the normal derivative of the solution to \eqref{a1}.

\subsubsection{A representation formula of the DN map}

We start by expressing the solution to \eqref{a1} in terms of the BSD.

\begin{lemma}
\label{lm-ru}
For all $f \in \sH$, the solution $u$ to \eqref{a1}, given by Proposition \ref{pr-eu}, reads 
\bel{p3a}
u(\cdot,t) = \sum_{n \geq1} u_n(t) \varphi_n\ \mbox{in}\ L^2(\Omega),\ t \in [0,T],
\ee
where 
$$ u_n(t) = - \int_{0}^t  e^{-\lambda_n s} \left( \int_{\pd \Omega} \psi_n(\sigma) f(\sigma,t-s) d \sigma \right) ds,\ n \in \N. $$
\end{lemma}
\begin{proof}
Since $\{ \varphi_n, \ n \in \N \}$ is an orthonormal basis of $L^2(\Omega)$ and $u \in C^1([0,T],L^2(\Omega))$ 
by Proposition \ref{pr-eu}, we have 
$$
 \forall s \in [0,T],\ u(\cdot,s) = \sum_{n \in \N} u_n(s) \varphi_n\ \mbox{in}\ L^2(\Omega),
$$
with
$$ s \mapsto u_n(s) = \int_{\Omega} u(x,s) \varphi_n(x) dx \in C^1([0,T]. $$
Moreover, for all $s \in [0,T]$, we get that
\beas
u_n'(s) & = &  \int_{\Omega} \pd_t u(x,s) \varphi_n(x) dx \\
& = & -\int_{\Omega} (-\Delta + q(x))u(x,t) \varphi_n(x) dx \\
& = & - \int_{\Omega} u(x,s) A_q \varphi_n(x) dx  + \int_{\pd \Omega} \pd_\nu u(\sigma,s) \varphi_n(\sigma) d \sigma
-  \int_{\pd \Omega} u(\sigma,s)  \pd_\nu  \varphi_n(\sigma) d \sigma,
\eeas
upon applying the Green formula. As a consequence, we have
\beas
u_n'(s) & = & - \lambda_n  \int_{\Omega} u(x,s) \varphi_n(x) dx  -  \int_{\pd \Omega} f(\sigma,s)  \psi_n(\sigma) d \sigma \\
& = & - \lambda_n u_n(s) -  \int_{\pd \Omega} f(\sigma,s)  \psi_n(\sigma) d \sigma,
\eeas
and hence
$e^{-\lambda_n s} \frac{d}{ds} ( e^{\lambda_n s} u_n(s) )
= u_n'(s) + \lambda_n u_n(s) =  -\int_{\pd \Omega} f(\sigma,s)  \psi_n(\sigma) d \sigma$. Thus, taking into account that $u_n(0)=0$, we get for every $t \in [0,T]$, that
$$ e^{\lambda_n t} u_n(t) = - \int_0^t e^{\lambda_n s}  \left( \int_{\pd \Omega} f(\sigma,t)  \psi_n(\sigma) d \sigma \right) ds, $$
upon integrating over $[0,t]$. 
\end{proof}
In general the series in \eqref{p3a} converges in $L^2(\Omega)$ but the convergence can be upgraded to $H^2(\Omega)$ by assuming that the function $f \in \sH$ satisfies the following condition
\bel{c1} 
\exists \eps \in (0,T_0),\ \forall (\sigma,t) \in \pd \Omega \times [T_0-\eps,T_0],\ f(\sigma,t)=0,
\ee
for some fixed $T_0 \in (0,T)$.
Indeed, in this case we deduce from the identity $u(\cdot,T_0)=f(\cdot,T_0)=0$ on $\pd \Omega$ that $u(\cdot,T_0) \in H^2(\Omega) \cap H_0^1(\Omega)=\mbox{dom}(A_q)$. As a consequence we have
$\sum_{n=1}^{+\infty} \lambda_n \abs{u_n(T_0)}^2 < \infty$ and the series in \eqref{p3a} converges in $H^2(\Omega)$ for $t=T_0$:
$$ u(\cdot,T_0)= \sum_{n=1}^{+\infty} u_n(T_0) \varphi_n\ \mbox{in}\ H^2(\Omega). $$
Thus, by continuity of the trace operator $g \mapsto (\pd_\nu g)_{| \pd \Omega}$ from $H^2(\Omega)$ to $H^{1 \slash 2}(\pd \Omega)$, we obtain that
\bea
\pd_\nu u(\cdot,T_0) & = & \sum_{n=1}^{+\infty} u_n(T_0) \psi_n \nonumber \\
& = & - \sum_{n=1}^{+\infty} \left( \int_{\eps}^{T_0} e^{-\lambda_n s} \left( \int_{\pd \Omega} f(\sigma,T_0-s) \psi_n(\sigma) d \sigma \right)  ds \right)\psi_n\ \mbox{in}\ H^{1 \slash 2}(\pd \Omega). \label{a2}
\eea

Having proved \eqref{a2}, we may now establish the:

\begin{lemma}
\label{lm-rndu}
For all $f \in \sH$ fulfilling \eqref{c1}, it holds true that
$$ \Lambda_q f = -\int_{\eps}^{T_0} \sum_{n=1}^{+\infty}  e^{-\lambda_n s} \left( \int_{\pd \Omega} f(\sigma,T_0-s) \psi_n(\sigma) d \sigma \right)  \psi_n ds\ \mbox{in}\ H^{1 \slash 2}(\pd \Omega).$$
\end{lemma}

\begin{proof}
In light of \eqref{a2}, we may write that 
\bel{a3}
u_n(T_0)= - \int_{\eps}^{T_0} e^{-\lambda_n s} \gamma_n(s) ds\ \mbox{with}\ \gamma_n(s) := \int_{\pd \Omega} f(\sigma,T_0-s) \psi_n(\sigma) d \sigma,\ s \in [\eps,T_0].
\ee
As $\abs{\gamma_n(s)} \leq  \norm{f(\cdot,T_0-s)}_{L^2(\pd \Omega)} \norm{\psi_n}_{L^2(\pd \Omega)}$, we get from \eqref{i2c} that
$$
\sum_{n=n_M}^{+\infty} e^{-\lambda_n s} \abs{\gamma_n(s)} \norm{\psi_n}_{H^{1 \slash 2}(\pd \Omega)} \leq c^2  \left( \sum_{n=1}^{+\infty} \lambda_n^2 e^{-\lambda_n \eps} \right) \norm{f(\cdot,T_0-s)}_{H^{3 \slash 2}(\pd \Omega)},\ s \in [\eps,T_0]. $$
Since $\sum_{n=1}^{+\infty} \lambda_n^2 e^{-\lambda_n \eps} <+\infty$ by \eqref{p2}, and $s \mapsto  \norm{f(\cdot,T_0-s)}_{H^{3 \slash 2}(\pd \Omega)} \in L^1(\eps,T_0)$, the Lebesgue dominated convergence theorem then yields
$$ \sum_{n=1}^{+\infty} \left( \int_{\eps}^{T_0}  e^{-\lambda_n s} \gamma_n(s) ds \right) \psi_n = \int_{\eps}^{T_0} \sum_{n=1}^{+\infty}  e^{-\lambda_n s} \gamma_n(s) \psi_n ds, $$
where the convergence of the series is considered in $H^{1 \slash 2}(\pd \Omega)$. 
\end{proof}

Armed with Lemma \ref{lm-rndu}, we turn now to proving Theorem \ref{thm-eq}.

\subsubsection{Proof of Theorem \ref{thm-eq}}
We split the proof into four steps. \\
\paragraph{\it First step: Set up.}
For $f \in \sHi$ obeying the condition \eqref{c1}, we apply Lemma \ref{lm-rndu} and get
$$ \Lambda_j(f) = -\int_{\eps}^{T_0} \sum_{n=1}^{+\infty}  e^{-\lambda_{j,n} s} \left( \int_{\Gi}   \psi_{j,n}(\sigma) f(\sigma,T_0-s) d \sigma \right)\psi_{j,n} ds\ \mbox{in}\ H^{1 \slash 2}(\Go),\ j=1,2. $$
This and the assumption $\Lambda_1(f)=\Lambda_2(f)$ for all $f \in \sHi$, yield for a.e. $\sigma' \in \Go$ that
\bel{a4}
\int_{\eps}^{T_0} \sum_{n=1}^{+\infty}   \left( \int_{\Gi} \left( e^{-\lambda_{1,n} s} \psi_{1,n}(\sigma) \psi_{1,n}(\sigma') - e^{-\lambda_{2,n} s} \psi_{2,n}(\sigma) \psi_{2,n}(\sigma') \right) f(\sigma,T_0-s) d \sigma \right) ds=0.
\ee
Let $g \in H^{3 \slash 2}(\pd \Omega)$ be such that $\supp\ g \subset \Gi$, fix $\eps \in (0,T_0)$ and pick $h \in C_0(0,T_0-\eps)$, the set of compactly supported continuous functions in $(0,T_0-\eps)$. 
Then, the function
$$ f(\sigma,t) :=\left\{ \begin{array}{cl} g(\sigma) h(t) & \mbox{if}\ (\sigma,t) \in \pd \Omega \times (0,T_0-\eps) \\ 0 & \mbox{if}\ (\sigma,t) \in \pd \Omega \times [T_0-\eps,T_0), \end{array} \right. $$
lies in $\sHi$ and verifies \eqref{c1}.
Thus, by applying \eqref{a4} with $f$ expressed in this form, we find for a.e. $\sigma' \in \Go$ that
$$
\int_{\eps}^{T_0} \sum_{n=1}^{+\infty}   \left( \int_{\Gi} \left( e^{-\lambda_{1,n} s} \psi_{1,n}(\sigma) \psi_{1,n}(\sigma') - e^{-\lambda_{2,n} s} \psi_{2,n}(\sigma) \psi_{2,n}(\sigma') \right) g(\sigma) d \sigma \right) h(T_0-s) ds=0.
$$
Since $h$ is arbitrary in $C_0(0,T_0-\eps)$, then $s \mapsto h(T_0-s)$ is arbitrary in $C_0(\eps,T_0)$, so the above identity yields
\bel{a4b}
\sum_{n=1}^{+\infty} \int_{\Gi} \left( e^{-\lambda_{1,n} s} \psi_{1,n}(\sigma) \psi_{1,n}(\sigma') - e^{-\lambda_{2,n} s} \psi_{2,n}(\sigma)
\psi_{2,n}(\sigma') \right) g(\sigma) d \sigma =0,\ s \in (\eps,T_0),\ \sigma' \in \Go,
\ee
by density of $C_0(\eps,T_0)$ in $L^1(\eps,T_0)$. This follows from the fact that
$$s \mapsto \sum_{n=1}^{+\infty} e^{-\lambda_{j,n} s} \left( \int_{\Gi}  \psi_{j,n}(\sigma) g(\sigma) d \sigma \right)\psi_{j,n} \in L^{\infty}((\eps,T_0), H^{1 \slash 2}(\Go)), $$
arising from the estimate 
 \beas
\sum_{n=n_M}^{+\infty} e^{-\lambda_{j,n} s} \left| \int_{\Gi}  \psi_{j,n}(\sigma) g(\sigma) d \sigma \right| \| \psi_{j,n} \|_{H^{1 \slash 2}(\Go)}
& \leq & \| g \|_{L^2(\Gi)}  \sum_{n=n_M}^{+\infty} e^{-\lambda_{j,n} \eps} \| \psi_{j,n} \|_{H^{1 \slash 2}(\pd \Omega)}^2 \\
& \leq & c^2 \| g \|_{H^{3 \slash 2}(\pd \Omega)} \left( \sum_{n=n_M}^{+\infty} \lambda_{j,n}^2 e^{-\lambda_{j,n} \eps} \right)< \infty,
\eeas
which is derived for all $s \in (\eps, T_0)$ from \eqref{i2c} and \eqref{p2}. Now, since $\eps$ is arbitrary in $(0,T_0)$, we infer from \eqref{a4b} that
\bel{a4c}
\sum_{n=1}^{+\infty} \int_{\Gi} \left( e^{-\lambda_{1,n} s} \psi_{1,n}(\sigma) \psi_{1,n}(\sigma') - e^{-\lambda_{2,n} s} \psi_{2,n}(\sigma)
\psi_{2,n}(\sigma') \right) g(\sigma) d \sigma =0,\ \sigma' \in \Go,\ s \in (0,T_0).
\ee
\\
\paragraph{\it Second step: Analytic continuation.} 
For $j=1,2$, we set
\bel{a5}
F_j(\sigma',s) = \sum_{n=1}^{+\infty} e^{-\lambda_{j,n} s} \left( \int_{\Gi} \psi_{j,n}(\sigma) g(\sigma) d \sigma \right) \psi_{j,n}(\sigma'),\ (\sigma',s) \in \Go \times (0,+\infty),
\ee
and we establish the:
\begin{lemma}
\label{lm-an}
The $H^{1 \slash 2}(\Go)$-function $s \mapsto F_j(\cdot,s)$, $j=1,2$, is analytic in $(0,+\infty)$.
\end{lemma}
\begin{proof}
For $K$, a compact subset of $(0,+\infty)$, we set $\eps:=\inf \{s,\ s \in K\}>0$. Then, with reference to \eqref{i2c}, we have 
\beas
e^{-\lambda_{j,n} s}  \abs{\int_{\Gi}  \psi_{j,n}(\sigma) g(\sigma) d \sigma} \norm{\psi_{j,n}}_{H^{1 \slash 2}(\pd \Omega)}
& \leq & e^{-\lambda_{j,n} \eps} \norm{\psi_{j,n}}_{L^2(\Gi)}  \norm{g}_{L^2(\Gi)} \norm{\psi_{j,n}}_{H^{1 \slash 2}(\pd \Omega)} \\
& \leq &
e^{-\lambda_{j,n} \eps} \norm{\psi_{j,n}}_{H^{1 \slash 2}(\pd \Omega)}^2 \norm{g}_{H^{3 \slash 2}(\pd \Omega)} \\
& \leq & c^2 \norm{g}_{\sH} e^{-\lambda_{j,n} \eps}  \lambda_{j,n}^2,\ s \in K,\ n \in \N,
\eeas
for $j=1,2$, where $c$ is the constant appearing in \eqref{p2}. 
Therefore, the series
$$\sum_{n=1}^{+\infty} e^{-\lambda_{j,n} s}  \left( \int_{\Gi}  \psi_{j,n}(\sigma) g(\sigma) d \sigma  \right) \psi_{j,n} $$ 
converges in $H^{1 \slash 2}(\pd \Omega)$, uniformly for $s \in K$. As a consequence, the mapping $s \mapsto F_j(\cdot,s)$ is analytic in $K$, since this is obviously the case for each function $s \mapsto e^{-\lambda_{j,n} s} ( \int_{\Gi}  \psi_{j,n}(\sigma) g(\sigma) d \sigma) \psi_{j,n}$ with $n \in \N$. Finally, $K$ being arbitrary in $(0,+\infty)$, we end up getting from this that $F_j$ is analytic in $(0,+\infty)$.
\end{proof}

Now, putting \eqref{a4c}-\eqref{a5} together, we infer from Lemma \ref{lm-an} that:
\bel{a6}
F_1(\cdot,s)=F_2(\cdot,s)\ \mbox{in}\ H^{1 \slash 2}(\Go),\ s \in (0,+\infty).
\ee
Next, for each $s \in (0,+\infty)$ fixed, we have
$$\sum_{n=n_M}^{+\infty} \lambda_{j,n} e^{-\lambda_{j,n} s} \norm{\psi_{j,n}}_{L^2(\Gi)} \leq c \sum_{n=n_M}^{+\infty} \lambda_{j,n}^2 e^{-\lambda_{j,n} s}  < \infty, $$ 
by \eqref{i2c} and \eqref{p2}, whence $\sigma \mapsto \sum_{n=1}^{+\infty} e^{-\lambda_{j,n} s} \abs{\psi_{j,n}(\sigma)} \norm{\psi_{j,n}}_{H^{1 \slash 2}(\Go)} \in L^2(\Gi)$. Therefore, it holds true that
$\sigma \mapsto \sum_{n=1}^{+\infty} e^{-\lambda_{j,n} s} \abs{\psi_{j,n}(\sigma) g(\sigma) } \norm{\psi_{j,n}}_{H^{1 \slash 2}(\Go)} \in L^1(\Gi)$ for all $s \in (0,+\infty)$.
Thus, by applying the Lebesgue dominated convergence theorem, we get that for every $s \in (0,+\infty)$,
$$ F_j(\cdot,s)=\sum_{n=1}^{+\infty} e^{-\lambda_{j,n} s} \left( \int_{\Gi}  \psi_{j,n}(\sigma) g(\sigma) d \sigma \right) \psi_{j,n} =  \int_{\Gi} \left( \sum_{n=1}^{+\infty}  e^{-\lambda_{j,n} s} \psi_{j,n}(\sigma)  \psi_{j,n}  \right) g(\sigma) d \sigma, $$
the convergence of the series being taken in the sense of $H^{1 \slash 2}(\Go)$. This and \eqref{a6} yield for a.e. $\sigma' \in \Go$ and all $s \in (0,+\infty)$, that
\bel{a7}
\int_{\Gi} \left( \sum_{n=1}^{+\infty}  e^{-\lambda_{1,n} s} \psi_{1,n}(\sigma) \psi_{1,n}(\sigma')   \right) g(\sigma)  d \sigma 
= \int_{\Gi} \left( \sum_{n=1}^{+\infty}  e^{-\lambda_{2,n} s} \psi_{2,n}(\sigma) \psi_{2,n} (\sigma')  \right) g(\sigma) d \sigma.
\ee
Moreover, since $\sigma \mapsto \sum_{n=1}^{+\infty}  e^{-\lambda_{j,n} s} \psi_{j,n}(\sigma) \psi_{j,n}  \in L^2(\Gi,H^{1 \slash 2}(\Go))$, $j=1,2$, and since $g$ is arbitrary in $H^{3 \slash 2}(\Gi)$, we deduce from \eqref{a7} and the density of $H^{3 \slash 2}(\Gi)$ in 
$L^2(\Gi)$, that for all $s \in (0,+\infty)$, the identity
\bel{a8}
\sum_{n=1}^{+\infty}  e^{-\lambda_{1,n} s} \psi_{1,n}(\sigma) \psi_{1,n}(\sigma') =  \sum_{n=1}^{+\infty}  e^{-\lambda_{2,n} s} \psi_{2,n}(\sigma) \psi_{2,n} (\sigma'),
\ee
holds in $L^2(\Gi,H^{1 \slash 2}(\Go))$, and consequently in $L^2(\Gi \times \Go)$.\\

\paragraph{\it Third step: Generalized Dirichlet series.}
Let $\{ \lambda_{j,n}',\ n \in \N \}$ be the sequence of strictly increasing eigenvalues of $A_j=A_{q_j}$, $j=1,2$. For each $n \in \N$, we denote by $m_{j,n}$ the geometric multiplicity\footnote{That is to say that $m_{j,n}$ is the dimension of the linear subspace $\ker(A_j - \lambda_{j,n}')$ in $L^2(\Omega)$.} of the eigenvalue $\lambda_{j,n}'$ and we introduce a family $\{ \varphi_{j,n,i},\ i=1,\ldots,m_{j,n} \}$ of eigenfunctions of $A_j$, which satisfy
$$ A_j \varphi_{j,n,i} = \lambda_{j,n}' \varphi_{j,n,i},\ i=1,\ldots,m_{j,n}, $$
and form a $L^2(\Omega)$-orthonormal basis of the eigenspace $\ker (A_j-\lambda_{j,n}')$. Next, we put 
\bel{a9} 
\theta_{j,n}(\sigma,\sigma') := \sum_{i=1}^{m_{j,n}} \psi_{j,n,i}(\sigma) \psi_{j,n,i}(\sigma'),\ (\sigma,\sigma') \in \Gi \times \Go,
\ee
where $\psi_{j,n,i}:=\pd_{\nu} \varphi_{j,n,i}$. For every fixed $s \in (0,+\infty)$, it is clear from \eqref{i2c} and \eqref{p2} that both series appearing in \eqref{a8} are absolutely convergent in $L^2(\Gi \times \Go)$, so we 
infer from \eqref{a9} that 
\bel{a9b}
\sum_{n=1}^{+\infty}  e^{-\lambda_{1,n}' s} \theta_{1,n}(\sigma,\sigma') =  \sum_{n=1}^{+\infty}  e^{-\lambda_{2,n}' s} \theta_{2,n}(\sigma,\sigma'),\ s \in (0,+\infty),\ (\sigma,\sigma') \in \Gi \times \Go.
\ee
Moreover, each function $\theta_{j,n}$, for $j=1,2$ and $n \in \N$, being not identically zero in $\Gi \times \Go$ according to Lemma \ref{lm-nz}, it follows from \eqref{a9b} and the standard theory of Dirichlet series, that
\bel{a10}
\lambda_{1,n}' = \lambda_{2,n}'\ \mbox{and}\ \theta_{1,n}=\theta_{2,n}\ \mbox{on}\ \Gi \times \Go,\ n \in \N.
\ee
\\
\paragraph{\it Fourth step: End of the proof.}
We are left with the task of showing that $m_{1,n}=m_{2,n}$ for all $n \in \N$, and that the eigenfunctions $\varphi_{2,n}$ can be chosen in such a way that
$$\psi_{1,n,i}=\psi_{2,n,i}\ \mbox{on}\ \Gi \cup \Go,\ i =1,\ldots,m_n, $$ 
where we have set $m_{j,n}:= m_{1,n}=m_{2,n}$.
Prior to doing so, we recall that for all non empty open subset $\Gamma \subset \pd \Omega$, the dimension of the subspace spanned by $\{ (\psi_{j,n,i})_{| \Gamma},\ i=1,\ldots,m_{j,n} \}$ in $L^2(\Gamma)$, is equal to $m_{j,n}$, i.e. that 
$$ m_{j,n}=\dim \{ \psi_{j,n,i},\ i=1,\ldots,m_{j,n} \},\ j=1,2. $$ 

a) We start by establishing that $m_{1,n}=m_{2,n}=m_n$ and that there exists $M_n \in \cO_{m_n}(\R)$, the set of  orthogonal matrices of size $m_n$, such that we have
\bel{a11}
\Psi_{2,n} = M_n \Psi_{1,n},
\ee
with $\Psi_{j,n}:=(\psi_{j,n,1},\ldots,\psi_{j,n,m_{j,n}})^T$, $j=1,2$,

To this end, we notice that the set $\Gamma_{n,1}:=\{ \sigma \in \Gi \cap \Go,\ \psi_{1,n,1}(\sigma) \neq 0 \}$ has positive Lebesgue measure, since $\psi_{1,n,1}$ is not identically zero in $\Gi \cap \Go$.
Similarly, the functions $\psi_{1,n,1}$ \and $\psi_{1,n,2}$ being linearly independent in $L^2(\Gi \cap \Go)$, the Lebesgue measure of the set 
$$ \Gamma_{n,2}:=\left\{ (\sigma_1,\sigma_2) \in (\Gi \cap \Go)^2,\ \det \left( \begin{array}{cc} \psi_{1,n,1}(\sigma_1) & \psi_{1,n,2}(\sigma_1) \\ \psi_{1,n,1}(\sigma_2) & \psi_{1,n,2}(\sigma_2) \end{array} \right) \neq 0 \right\} $$
is positive\footnote{Otherwise, we would have $\psi_{1,n,2}(\sigma_1) \psi_{1,n,1} - \psi_{1,n,1}(\sigma_1) \psi_{1,n,2}=0$ in $L^2(\Gi \cap \Go)$ for a.e. $\sigma_1 \in \Gi \cap \Go$, and hence $\psi_{1,n,1}(\sigma_1)=0$ since $\psi_{1,n,1}$ and $\psi_{1,n,2}$ are linearly independent in $L^2(\Gamma)$, which is a contradiction with the fact that $\Gamma_{n,1}$ has non-zero Lebesgue measure.}.
Thus, by induction on $i$, we can build a subset $\Gamma_{n,m_{1,n}} \subset (\Gi \cap \Go)^{m_{1,n}}$, with positive Lebesgue measure, such that following matrix
$$ P_{1,n}(\boldsymbol{\sigma}) := \left( \begin{array}{ccc} \psi_{1,n,1}(\sigma_1) & \ldots & \psi_{1,n,m_{1,n}}(\sigma_1) \\ \vdots & & \vdots \\ \psi_{1,n,1}(\sigma_{m_{1,n}}) & \ldots & \psi_{1,n,m_{1,n}}(\sigma_{m_{1,n}}) \end{array} \right) $$
is invertible for a.e. $\boldsymbol{\sigma}:=(\sigma_1,\ldots,\sigma_{m_{1,n}}) \in \Gamma_{n,m_{1,n}}$. 

Next, with reference to \eqref{a10}, we get upon applying \eqref{a9} with $\sigma'=\sigma_j$ for $j=1,\ldots,m_{1,n}$, that
$$
\sum_{i=1}^{m_{1,n}} \psi_{1,n,i}(\sigma_j) \psi_{1,n,i}(\sigma) = \sum_{i=1}^{m_{2,n}} \psi_{2,n,i}(\sigma_j) \psi_{2,n,i}(\sigma),\ \sigma
\in \Gi,\ \boldsymbol{\sigma}=(\sigma_1,\ldots,\sigma_{m_{1,n}}) \in \Gamma_{n,m_{1,n}}. $$
This can be equivalently rewritten as
$P_{1,n}(\boldsymbol{\sigma}) \Psi_{1,n}(\sigma) = P_{2,n}(\boldsymbol{\sigma}) \Psi_{2,n}(\sigma)$ for a.e. $\sigma \in \Gi$,  where $P_{2,n}(\boldsymbol{\sigma})$ is the following $m_{1,n} \times m_{2,n}$ matrix:
$$ P_{2,n}(\boldsymbol{\sigma}) := \left( \begin{array}{ccc} \psi_{2,n,1}(\sigma_1) & \ldots & \psi_{2,n,m_{2,n}}(\sigma_1) \\ \vdots & & \vdots \\ \psi_{2,n,1}(\sigma_{m_{1,n}}) & \ldots & \psi_{2,n,m_{2,n}}(\sigma_{m_{1,n}}) \end{array} \right).  $$
Therefore, putting $M_n(\boldsymbol{\sigma}):=P_{1,n}(\boldsymbol{\sigma})^{-1} P_{2,n}(\boldsymbol{\sigma})$ for a.e. $\boldsymbol{\sigma} \in \Gamma_{n,m_{1,n}}$, we get that $\Psi_{1,n}(\sigma) = M_n(\boldsymbol{\sigma}) \Psi_{2,n}(\sigma)$ for a.e. $\sigma \in \Gi$.
Further, taking $\sigma=\sigma_j$ in \eqref{a10}, we get in the same way as before that $\Psi_{1,n}(\sigma') = M_n(\boldsymbol{\sigma}) \Psi_{2,n}(\sigma')$ for a.e. $\sigma' \in \Go$. As a consequence, we have 
\bel{a12}
\Psi_{1,n}(\sigma) = M_n(\boldsymbol{\sigma}) \Psi_{2,n}(\sigma),\ \sigma \in \Gi \cup \Go,\ \boldsymbol{\sigma} \in \Gamma_{n,m_{1,n}}. \ee
Since $\dim \{ \psi_{j,n,i},\ i=1,\ldots,m_{j,n} \} = m_{j,n}$ in $L^2(\Gi \cup \Go)$, $j=1,2$, we infer from \eqref{a12} that $m_{1,n} \leq m_{2,n}$. Moreover, as $j=1$ and $j=2$ play symmetric roles here, we have $m_{2,n} \leq m_{1,n}$, so we end up getting that $m_{1,n}=m_{2,n}$.

It remains to show that $M_n(\boldsymbol{\sigma}) \in \cO_{m_n}(\R)$ for a.e. $\boldsymbol{\sigma} \in \Gamma_{n,m_n}$.
This can be done by plugging each of the two following equalities $\Psi_{1,n}(\sigma) = M_n(\boldsymbol{\sigma}) \Psi_{2,n}(\sigma)$ for a.e. $\sigma \in \Gi$ and $\Psi_{1,n}(\sigma') = M_n(\boldsymbol{\sigma})  \Psi_{2,n}(\sigma')$ for a.e. $\sigma' \in \Go$, in \eqref{a9b}. We obtain that $M_n(\boldsymbol{\sigma})  \psi_{2,n}(\sigma) \cdot M_n(\boldsymbol{\sigma})  \psi_{2,n}(\sigma') = \psi_{2,n}(\sigma) \cdot \psi_{2,n}(\sigma')$, where the symbol $\cdot$ stands for the Euclidian scalar product in $\R^{m_n}$. Therefore, we have $(M_n(\boldsymbol{\sigma})^T M_n(\boldsymbol{\sigma})  - I_{m_n})  \psi_{2,n}(\sigma) \cdot \psi_{2,n}(\sigma') = 0$ for a.e. $(\sigma,\sigma') \in \Gi \times \Go$, where $I_{m_n}$ denotes the identity matrix of size $m_n$. The family $\{ \psi_{2,n,i},\ i=1,\ldots,m_n \}$, being linearly independent in $L^2(\Go)$, this entails that $(M_n(\boldsymbol{\sigma})^T M_n(\boldsymbol{\sigma})  - I_{m_n})  \psi_{2,n}(\sigma)=0$ for a.e. $\sigma \in \Gi$. Similarly, using that $\{ \psi_{2,n,i},\ i=1,\ldots,m_n \}$ is linearly independent in $L^2(\Gi)$, we get that $M_n(\boldsymbol{\sigma})^T M_n(\boldsymbol{\sigma})  - I_{m_n}=0$, which establishes that $M_n(\boldsymbol{\sigma}) \in \cO_{m_n}(\R)$.\\

b) We turn now to showing that $\psi_{1,n,i}=\psi_{2,n,i}$ on $\pd \Omega$ for all $i \in \{ 1, \ldots, m_n \}$, up to some appropriate choice of the eigenfunctions $\varphi_{2,n,i}$. To do that, we write $\varphi_{2,n}=(\varphi_{2,n,1},\ldots,\varphi_{2,n,m_n})^T$ and we consider
$$ \varphi_{2,n}' =(\varphi_{2,n,1}',\ldots,\varphi_{2,n,m_n}')^T := M_n(\boldsymbol{\sigma})^T \varphi_{2,n}, $$
where $\boldsymbol{\sigma}$ is arbitrary in $\Gamma_{n,m_n}$. 

Writing $M_n$ instead of $M_n(\boldsymbol{\sigma})$ in the sequel, we have for all $(i,k) \in \{ 1, \ldots , m_n \}^2$ that
$$ \langle \varphi_{2,n,i}' , \varphi_{2,n,k}' \rangle_{L^2(\Omega)} = \sum_{r,s=1}^{m_n} (M_n^T)_{ir} (M_n^T)_{ks} \langle \varphi_{2,n,r} , \varphi_{2,n,s}  \rangle_{L^2(\Omega)}= \sum_{r,s=1}^{m_n} (M_n)_{ri} (M_n)_{sk} \delta_{rs},
$$
where $\delta$ denotes the Kronecker symbol\footnote{That is to say that
$$ \delta_{rs}:= \left\{ \begin{array}{ll} 1 & \mbox{if}\ r=s \\ 0 & \mbox{otherwise}. \end{array} \right.$$}.
Thus it holds true for all $(i,k) \in \{ 1, \ldots , m_n \}^2$ that
$\langle \varphi_{2,n,i}' , \varphi_{2,n,k}' \rangle_{L^2(\Omega)}  = \sum_{r=1}^{m_n} (M_n)_{ri} (M_n)_{rk}= \sum_{r=1}^{m_n} (M_n^T)_{ir} (M_n)_{rk}= (M_n^T M_n)_{ik}=\delta_{ik}$. Consequently, the family
$\{  \varphi_{2,n,i}',\ i=1,\ldots,m_n \}$ is orthonormal in $L^2(\Omega)$. Moreover, for all $i \in \{ 1, \ldots, m_n \}$ and for a.e. $\sigma \in \pd \Omega$, we get upon writing $\nu(\sigma)=(\nu_1(\sigma),\ldots,\nu_d(\sigma))^T$, that
$$
\psi_{2,n,i}'(\sigma) := \nabla \varphi_{2,n,i}'(\sigma) \cdot \nu(\sigma) = \sum_{\ell=1}^d \partial_{\ell} \varphi_{2,n,i}'(\sigma) \nu_{\ell}(\sigma) $$
reads
\beas
\psi_{2,n,i}'(\sigma) & = & \sum_{\ell=1}^d \partial_{\ell} \left( \sum_{r=1}^{m_n} (M_n^T)_{ir} \varphi_{2,n,r}(\sigma) \right) \nu_{\ell}(\sigma) \nonumber \\
& = &   \sum_{r=1}^{m_n} (M_n)_{ri} \left( \sum_{\ell=1}^d \partial_{\ell} \varphi_{2,n,r}(\sigma)  \nu_{\ell}(\sigma) \right) \nonumber \\
& =  & \sum_{r=1}^{m_n} (M_n)_{ri} \psi_{2,n,r}(\sigma).
\eeas
Therefore, we have $\Psi_{2,n}'=M_n^T \Psi_{2,n}$ and hence $\Psi_{2,n}'=\Psi_{1,n}$ on $\pd \Omega$, by virtue of \eqref{a11}.
This terminates the proof of Theorem \ref{thm-eq}.

\bigskip
\noindent {\bf Acknowledgements.} This work was partially supported by the Agence Nationale de la Recherche under grant ANR-17- CE40-0029.

\bigskip


\noindent {\sc \'Eric Soccorsi}, Aix-Marseille Universit\'e, CNRS, CPT, Marseille, France.\\
E-mail: {\tt eric.soccorsi@univ-amu.fr}.

\end{document}